\documentclass[12pt,reqno]{amsart}

\usepackage[margin=1in]{geometry}
\usepackage{amsmath,amsfonts,amsthm,amssymb,amscd, dsfont,tipa,mathtools,listings, rotating, multirow, bbm}
\usepackage{bm}
\usepackage{color}
\usepackage{comment}
\usepackage{seqsplit}
\usepackage{enumitem}
\usepackage{soul}
\usepackage{mathrsfs}
\usepackage[titletoc,title]{appendix}
\usepackage{tikz-cd}
\usepackage{lineno}
\usepackage{lipsum}
\usepackage{hyperref}
\usepackage{constants}
\newconstantfamily{abcon}{symbol=\kappa}
\usepackage{soul}

\numberwithin{equation}{section}

\newcommand{\bburl}[1]{\textcolor{blue}{\url{#1}}}

\newcommand\Item[1][]{%
  \ifx\relax#1\relax  \item \else \item[#1] \fi\abovedisplayskip=0pt\abovedisplayshortskip=0pt~\vspace*{-\baselineskip}}

\newcommand\be{\begin{equation}}
	\newcommand\ee{\end{equation}}
\newcommand\bi{\begin{itemize}}
	\newcommand\ei{\end{itemize}}
\newcommand\ben{\begin{enumerate}}
	\newcommand\een{\end{enumerate}}

\newtheorem{thm}{Theorem}[section]

\newtheorem{cor}[thm]{Corollary}
\newtheorem{lem}[thm]{Lemma}
\newtheorem{rem}[thm]{Remark}

\newtheorem*{hypoh}{Hypothesis $\bm{\mathrm{H}}_{\pi}$}
\newtheorem*{hypoA}{Hypothesis $\bm{\mathrm{A}}$}
\newtheorem*{hypoV}{Hypothesis $\bm{\mathrm{V}}$ (vector uniqueness)} % NEW

\newtheorem{prop}[thm]{Proposition}
\newtheorem{exa}[thm]{Example}

\theoremstyle{remark}

\newcommand{\bs}{\backslash}

\DeclareMathOperator{\GL}{GL}

\DeclareMathOperator{\Real}{Re}
\DeclareMathOperator{\Imag}{Im}

\DeclareMathOperator{\N}{\mathbb{N}}

\DeclareMathOperator{\der}{der}
\DeclareMathOperator{\cond}{cond}
\makeatletter
\let\@wraptoccontribs\wraptoccontribs
\makeatother

\begin{document}
\title[Zeros of Polynomials in Derivatives of Automorphic $L$-functions]{Zeros of Polynomials in Derivatives of Automorphic $L$-functions 
}
\author{Anji Dong}
\address{Anji Dong: 
Department of Mathematics, University of Illinois Urbana-Champaign, Altgeld hall, 1409 W. Green Street, Urbana, IL, 61801, USA}
\email{anjid2@illinois.edu}

\author{Nawapan Wattanawanichkul}
\address{Nawapan Wattanawanichkul: Department of Mathematics, University of Illinois Urbana-Champaign, Altgeld hall, 1409 W. Green Street, Urbana, IL, 61801, USA}
\email{nawapanwattanawanichkul4@gmail.com}
    
\author{Alexandru Zaharescu}
\address{Alexandru Zaharescu: Department of Mathematics, University of Illinois Urbana-Champaign, Altgeld hall, 1409 W. Green Street, Urbana, IL, 61801, USA \and Simion Stoilow Institute of Mathematics  of the Romanian Academy, P.O. Box 1-764, RO-014700 Bucharest, Romania}
\email{zaharesc@illinois.edu}

\begin{abstract}
 Let $\mathfrak{F}_m$ be the set of all cuspidal automorphic representations of $\mathrm{GL}_m(\mathbb{A}_{\mathbb{Q}})$, and let $F(s,\bm{\pi})$ be a polynomial in the derivatives of $L$-functions associated with representations $\pi  \in \bigcup_{m=1}^{\infty} \mathfrak{F}_m$. We establish an asymptotic formula for the number of nontrivial zeros of $F(s,\bm{\pi})$ with $0 < \operatorname{Im}(s) < T$. We explicitly determine the main term of this formula in terms of the dimensions, the arithmetic conductors, and the orders of differentiation of the component $L$-functions. Furthermore, we show that, under certain conditions, almost all nontrivial zeros of $F(s,\bm{\pi})$ lie near the critical line $\operatorname{Re}(s)=1/2$.
\end{abstract}
\maketitle
\setcounter{tocdepth}{1}

\tableofcontents
\section{Introduction and summary of main results}\label{sec:intro}

\subsection{Notation and setup}\label{subsec:notation} A central problem in number theory is to understand the distribution of zeros of the Riemann zeta function $\zeta(s)$. It is known that 
 $\zeta(s)$ does not have zeros in the half-plane $\Real(s) \ge 1$, and its only zeros in the half-plane $\Real(s) \le 0$ are at negative even integers, known as the \emph{trivial zeros}. Consequently, all other zeros, known as the \textit{nontrivial zeros}, lie within the critical strip,
$0 < \Real(s) < 1$.  The Riemann Hypothesis (RH) asserts that all nontrivial zeros of  $\zeta(s)$  lie on the critical line $\Real(s) = 1/2$, a conjecture that remains unproven. 
By contrast, the global distribution of these zeros is well understood; the number of nontrivial zeros $\zeta(s)$, with $0 < \Imag(s) < T$, denoted by $N(0,T)$ and counted with multiplicity, satisfies the classical asymptotic formula
\[
    N(0,T) = \frac{T}{2\pi}\log\frac{T}{2\pi e} + O(\log T).
\]

Motivated by the classical result, our first goal in this paper is to investigate the analogue of the above formula within a much broader setting.
Let $\mathbb{A}_{\mathbb{Q}}$ be the ring of ad\`{e}les over $\mathbb{Q}$. For any integer $m \ge 1$, let $\mathfrak{F}_m$ be the set of all cuspidal automorphic representations of $\mathrm{GL}_m(\mathbb{A}_{\mathbb{Q}})$ whose central characters are unitary and normalized to be trivial on the diagonally embedded positive reals. 
 
 We let $\mathfrak{S}$ denote the set of all cuspidal automorphic $L$-functions $L(s,\pi)$ associated with  $\pi \in \mathfrak{F}_{m}$ for $m\in\N$, that is, 
 \[
 \mathfrak{S} = \bigcup_{m=1}^{\infty}\bigcup_{\pi \in \mathfrak{F}_m} \{L(s,\pi)\}.
 \]
 We define $\mathcal{A}$ to be the smallest $\mathbb{C}$-algebra that contains $\mathfrak{S}$. The elements of $\mathcal{A}$ are polynomials in finitely many $L$-functions in $\mathfrak{S}$ with complex coefficients, that is, 
\[
\mathcal{A} = \Big\{ \sum_{j=1}^{M} c_{j} \prod_{u=1}^V L(s,\pi_u)^{d_{u,j}} : M,V \in  \mathbb{N},\, c_{j} \in \mathbb{C}\bs\{0\},\, d_{u,j} \in \mathbb{N}\cup\{0\},\ L(s,\pi_u)\in\mathfrak{S}  \Big\}.
\]
Furthermore, we define $\mathcal{B}$ to be the smallest $\mathbb{C}$-algebra containing $\mathcal{A}$ that is also closed under \textit{differentiation} (with respect to the complex variable $s$). For any integer $l \ge 0$, let $L^{(l)}(s, \pi_u)$ denote the $l\text{-th}$ derivative of $L(s,\pi_u)$.
Thus,
\[
\mathcal{B}=\Big\{ \sum_{j=1}^{M} c_j  \prod_{u=1}^V \prod_{l=0}^{k_u}  L^{(l)}(s, \pi_u)^{d_{u,l,j}}:M,V \in  \mathbb{N},\, c_{j} \in \mathbb{C}\bs\{0\},\, k_u,d_{u,l,j} \in \mathbb{N}\cup\{0\},\ L(s,\pi_u)\in\mathfrak{S}\Big\}.
\]

{In particular, let $\bm{\pi} = (\pi_1, \dots, \pi_V)$ be a fixed vector of $V$ pairwise non-isomorphic representations in $\mathfrak{S}$. By definition, every $F(s,\bm{\pi}) \in \mathcal{B}$ admits \emph{at least one} expression as a linear combination of monomial products in the variables $L^{(l)}(s, \pi_u)$, where $l \ge 0$ and $1 \le u \le V$: there exist an integer $M > 0$, integers $k_u, d_{u,l,j} \ge 0$, and nonzero complex coefficients $c_{j}$, for pairwise distinct exponent tuples $(d_{u,l,j})_{u,l}$ indexed by $j$, such that
\begin{equation}\label{def:F}
F(s,\bm{\pi}) = \sum_{j=1}^{M} c_j  \prod_{u=1}^V \prod_{l=0}^{k_u}  L^{(l)}(s, \pi_u)^{d_{u,l,j}} = \sum_{j=1}^{M}  F_j(s,\bm{\pi}).
\end{equation}
Throughout the paper, ``$F(s,\bm{\pi}) \in \mathcal{B}$ as in \eqref{def:F}''
should be read as fixing both the function $F(s,\bm{\pi})$ and a chosen
representation of the form \eqref{def:F}; every degree invariant, index set,
and hypothesis introduced below is computed from this chosen representation,
and our theorems describe the analytic behavior of $F(s,\bm{\pi})$ relative
to it. 
}

To state our main results, we first define several key quantities that measure in a suitable way the complexity of $F(s,\bm{\pi})$ with respect to the dimensions, orders of differentiation, and arithmetic conductors of its component $L$-functions. Let $m_u$ denote the dimension (or degree) of each $L(s,\pi_u)$, implying $\pi_u \in \mathfrak{F}_{m_u}$, and let $\mathfrak{q}_{\pi_u}$ denote the arithmetic conductor of $L(s,\pi_u)$. (See Section \ref{sec:L-functions} for more details.) For each monomial $F_j(s,\bm{\pi})$ as defined in \eqref{def:F}, we denote
\begin{alignat*}{2}
&\deg_{\dim}(F_j(s,\bm{\pi})) &&:= \sum_{u=1}^V m_u\sum_{l=0}^{k_u} d_{u,l,j},\\
&\deg_{\der}(F_j(s,\bm{\pi})) &&:= \sum_{u=1}^V \sum_{l=0}^{k_u} ld_{u,l,j},\\
 \intertext{and}
& \deg_{\cond}(F_j(s,\bm{\pi})) \hspace{0.02in} &&:= \sum_{u=1}^V \log \mathfrak{q}_{\pi_u} \sum_{l=0}^{k_u} d_{u,l,j}.
\end{alignat*}
Furthermore, we define the following degrees for the whole polynomial $F(s,\bm{\pi})$:
\begin{alignat}{2}
&\deg_{\dim}(F(s,\bm{\pi})) &&:= \max_{1 \le j \le M} \big\{\deg_{\dim}(F_j(s,\bm{\pi}))\big\},\label{def:degm,D(F)}\\
&\deg_{\mathrm{der}}(F(s,\bm{\pi})) 
 &&:=\max_{1 \le j \le M}\big\{ \deg_{\mathrm{der}}(F_j(s,\bm{\pi})) :
\deg_{\dim}(F_j(s,\bm{\pi})) = \deg_{\dim}(F(s,\bm{\pi}))\big\},\\
&\deg_{\mathrm{cond}}(F(s,\bm{\pi})) 
&&:= 
\max_{1 \le j \le M}\Big\{ \deg_{\cond}(F_j(s,\bm{\pi})) :
\begin{array}{c}
\begin{aligned}
    \deg_{\dim}(F_j(s,\bm{\pi})) &= \deg_{\dim}(F(s,\bm{\pi})),\\ 
    \deg_{\mathrm{der}}(F_j(s,\bm{\pi})) &= \deg_{\mathrm{der}}(F(s,\bm{\pi}))
\end{aligned}
\end{array}
\Big\},\label{def:deg cond,D(F)}\\
 &\deg'_{\mathrm{cond}}(F(s,\bm{\pi})) 
 &&:=\max_{1 \le j \le M}\big\{ \deg_{\mathrm{cond}}(F_j(s,\bm{\pi})) :
    \deg_{\dim}(F_j(s,\bm{\pi})) = \deg_{\dim}(F(s,\bm{\pi}))\big\}, \label{def:deg dimD(F)-horizon}\\
  &\deg'_{\mathrm{der}}(F(s,\bm{\pi})) 
   &&:= 
    \max_{1 \le j \le M}\Big\{ \deg_{\der}(F_j(s,\bm{\pi})) :
    \begin{array}{c}
    	\begin{aligned}
    		\deg_{\dim}(F_j(s,\bm{\pi})) &= \deg_{\dim}(F(s,\bm{\pi})),\\ 
    		\deg_{\mathrm{cond}}(F_j(s,\bm{\pi})) &= \deg'_{\mathrm{cond}}(F(s,\bm{\pi}))
\end{aligned}
\end{array}
\Big\}.\label{def:deg cond,D(F)-horizon}
\end{alignat}

{ At this stage, it is not yet clear that these degree invariants are well-defined independently of the choice of representation \eqref{def:F}. After proving Lemma \ref{lem:FE}, we will see that the degree invariants appearing in our main results are intrinsic to $F(s,\bm{\pi})$ and hence independent of this choice.
}

For sufficiently large $\Real(s)$, the function $F(s,\bm{\pi})$ is given by a Dirichlet series, 
\begin{equation}\label{eq:Dirichlet series}
    F(s,\bm{\pi}) = \sum_{n=n_F}^{\infty} \frac{\eta_n}{n^{s}},\ \textit{where $n_F$ denotes the first index with a nonzero coefficient.}
\end{equation}
 Let $J$ be the set of \mbox{indices $j$} for which the $j$-th term $F_j(s, \bm{\pi})$ achieves the maximal degrees in \eqref{def:degm,D(F)}--\eqref{def:deg cond,D(F)}. In other words, 
\begin{equation}\label{def:J}
J := \left\{ 1\le j \le M :
    \begin{array}{c}
        \deg_{\dim}(F_j(s,\bm{\pi}))  = \deg_{\dim}(F(s,\bm{\pi})), \\
        \deg_{\der}(F_j(s,\bm{\pi}))
         = \deg_{\mathrm{der}}(F(s,\bm{\pi})),\\
        \deg_{\cond}(F_j(s,\bm{\pi}))  = \deg_{\mathrm{cond}}(F(s,\bm{\pi})).
    \end{array}
\right\}.
\end{equation}
 Let $J'$ be the set of \mbox{indices $j$} for which the $j$-th term $F_j(s, \bm{\pi})$ achieves the maximal degrees in \eqref{def:degm,D(F)}, \eqref{def:deg dimD(F)-horizon}, and \eqref{def:deg cond,D(F)-horizon}. In other words, 
\begin{equation}\label{def:J'}
	J' := \left\{ 1\le j \le M :
	\begin{array}{c}
		\deg_{\dim}(F_j(s,\bm{\pi}))  = \deg_{\dim}(F(s,\bm{\pi})), \\
		\deg_{\cond}(F_j(s,\bm{\pi}))
		= \deg'_{\mathrm{cond}}(F(s,\bm{\pi})),\\
		\deg_{\der}(F_j(s,\bm{\pi}))  = \deg'_{\mathrm{der}}(F(s,\bm{\pi})).
	\end{array}
	\right\}.
\end{equation}

{In what follows, we introduce two conditions on $F(s,\bm{\pi})$, which we state together and
cite, by name, in the hypothesis line of every result that uses them.

For $1 \le u \le V$ and $1 \le j \le M$, write $e_{u,j} := \sum_{l=0}^{k_u} d_{u,l,j}$
and $\ell_{u,j} := \sum_{l=0}^{k_u} l\, d_{u,l,j}$, and call
\[
\mathbf{e}(j) := \big( (e_{u,j}, \ell_{u,j}) \big)_{u=1}^{V} \in (\mathbb{Z}_{\ge 0} \times \mathbb{Z}_{\ge 0})^{V}
\]
the \emph{type vector} of the monomial $F_j(s,\bm{\pi})$. 

\begin{hypoV}
Every $j \in J$ shares the same type vector $\mathbf{e}(j)$, and every $j \in J'$ shares the same
type vector $\mathbf{e}(j)$. (The common vector for $J$ need not coincide with the common vector
for $J'$.)
\end{hypoV}

Under Hypothesis $\bm{\mathrm{V}}$, the quantity $\prod_{u=1}^V m_u^{\ell_{u,j}}$ is the same
value for every $j\in J$ (resp.\ $J'$), so it may be absorbed as a common nonzero constant, and
we state the non-destruction condition in the following form:
\begin{align}\label{assump:assumption 1}
    \sum_{j \in J} c_j \neq 0 \quad \text{and} \quad \sum_{j \in J'} c_j \neq 0.
\end{align}

\begin{rem}
        Condition \eqref{assump:assumption 1} precludes destructive cancellation of the
        coefficient sum among the dominant monomials indexed by $J$ and $J'$; Hypothesis
        $\bm{\mathrm{V}}$ is what makes this the right thing to check, since it guarantees those
        monomials reduce to a single common function of $s$ at leading order (Lemma
        \ref{lem:FE}), so that \eqref{assump:assumption 1} is the only remaining obstruction to
        cancellation. Without Hypothesis $\bm{\mathrm{V}}$, monomials tied in the aggregate
        degrees \eqref{def:degm,D(F)}--\eqref{def:deg cond,D(F)-horizon} can correspond to
        \emph{different} type vectors, i.e.\ genuinely different functions of comparable size,
        and \eqref{assump:assumption 1} alone would no longer preclude destructive interference
        between them. This is why the two conditions are always invoked together.
    \end{rem}
        
  \begin{rem}      
        We note that Hypothesis $\bm{\mathrm{V}}$ does not
        force $J=J'$: the two sets break ties among the $\deg_{\dim}$-maximal monomials in
        different orders ($\deg_{\mathrm{der}}$ before $\deg_{\mathrm{cond}}$ for $J$, versus
        $\deg_{\mathrm{cond}}'$ before $\deg_{\mathrm{der}}'$ for $J'$), and generically select
        different type vectors. For instance, $F(s)=\zeta''(s)+L(s,\chi_4):=F_1(s)+F_2(s)$ has $J=\{1\}$
        but $J'=\{2\}$. Both index sets are necessary: $J'$ governs the
        $\sigma\to\infty$ behavior used in Proposition \ref{prop:zfr1} and Lemma
        \ref{thm:zfr2}, while $J$ governs the $|t|\to\infty$ behavior used in the proof of
        Theorem \ref{thm:number of zeros} and in \S\ref{sec:zeros near critical line}. 
\end{rem}}

\subsection{Main results}
Analogous to the Riemann zeta function, the \textit{trivial zeros} of $F(s,\bm{\pi})$ are those zeros located in predictable neighborhoods, which arise from the Gamma factors of the component $L$-functions. The remaining zeros of $F(s,\bm{\pi})$ are the \textit{nontrivial zeros}, which lie in a vertical strip with bounded real parts. Precise neighborhoods of these trivial zeros and the existence of a vertical strip confining the nontrivial zeros are shown in Proposition \ref{prop:zfr1} and Lemma \ref{thm:zfr2}. 

We define the counting function $N_{F(s,\bm{\pi})}(0,T)$ as the number of nontrivial zeros of $F(s,\bm{\pi})$ with positive imaginary part up to $T$, that is,
\[
N_{F(s,\bm{\pi})}(0,T)= \#\{\text{nontrivial zeros of } F(s,\bm{\pi}) \text{ with } 0 < \Imag(s) < T \}.
\]
Our first theorem establishes an asymptotic formula for $N_{F(s,\bm{\pi})}(0,T)$.
\begin{thm}\label{thm:number of zeros}
Let $ F(s,\bm{\pi}) \in \mathcal{B}$ be defined as in \eqref{def:F} satisfying \eqref{assump:assumption 1} and Hypothesis $\bm{\mathrm{V}}$. Then there exist constants $\alpha_1 := \alpha_1(F(s,\bm{\pi}))$ and $\alpha_2 := \alpha_2(F(s,\bm{\pi}))$ such that for $T \ge 2$,
    \begin{align}
    N_{F(s,\bm{\pi})}(0,T) = \alpha_1 T \log T  + \alpha_2 T + O_F(\log T),\label{eq:eq in theorem 1.2}
    \end{align}
     where the implied constant on the right side of \eqref{eq:eq in theorem 1.2} depends only on $F$.
     Furthermore, the constants $\alpha_1$ and $\alpha_2$ are given explicitly as follows:
        \begin{align*}
            \alpha_1 &= \frac{1}{2\pi}\deg_{\dim}(F(s,\bm{\pi})), \\
             \intertext{and}
            \alpha_2 &= \frac{1}{2\pi}\Big(\deg_{\mathrm{cond}}(F(s,\bm{\pi}))-\deg_{\dim}(F(s,\bm{\pi}))\log(2\pi e)-\log n_F\Big).
        \end{align*}
\end{thm}

{\begin{rem}
By Remark \ref{rem:intrinsic} established below, $\deg_{\dim}(F(s,\bm{\pi}))$ and $\deg_{\mathrm{cond}}(F(s,\bm{\pi}))$ are known to be intrinsic to $F(s,\bm{\pi})$, independently of the representation \eqref{def:F} used to compute them. Likewise, $n_F$ is intrinsic (it is read off the Dirichlet series coefficients of $F(s,\bm{\pi})$ itself, via \eqref{eq:Dirichlet series}). So although we do not assume the representation \eqref{def:F} is unique, the conclusion of Theorem \ref{thm:number of zeros} does not depend on which representation is used to establish it. 
\end{rem}}

{
\begin{exa}[A worked example with two distinct $L$-functions]
Let $\chi_4$ denote the primitive Dirichlet character modulo $4$, and take
\[
F(s)
=
2\zeta(s)\zeta''(s)L'(s,\chi_4)^2
+
\zeta'(s)^2L(s,\chi_4)L''(s,\chi_4).
\]
Taking $u=1$ for $\zeta(s)$ and $u=2$ for $L(s,\chi_4)$, we have $F_1(s)=2\zeta(s)\zeta''(s)L'(s,\chi_4)^2$ has type vectors $(e_1,\ell_1)=(2,2)$ and $(e_2,\ell_2)=(2,2)$,
while  $F_2(s)=\zeta'(s)^2L(s,\chi_4)L''(s,\chi_4)$
has exactly the same type vectors. Thus Hypothesis $\bm{\mathrm{V}}$ holds. Moreover,
\[
\deg_{\dim}(F)= 4,\qquad
\deg_{\mathrm{der}}(F)=4,\qquad
\deg_{\mathrm{cond}}(F)=2\log4,
\]
with all three degrees attained by both monomials; hence,
$J=J'=\{1,2\}$. Since $\sum_{j\in J}
c_j = 3$ and $\sum_{j\in J'}
c_j = 3$, the condition \eqref{assump:assumption 1}
holds as well. A direct calculation of the Dirichlet coefficients shows that $n_F=12$. Therefore,
Theorem \ref{thm:number of zeros} gives
\[
\alpha_1=\frac{4}{2\pi}=\frac{2}{\pi},
\qquad
\alpha_2
=
\frac{1}{2\pi}
\bigl(2\log 4-4\log(2\pi e)-\log 12\bigr).
\]
\end{exa}
}

Returning to the distribution of zeros of $\zeta(s)$, the classical work of Bohr and Landau \cite{bohr} shows that almost all nontrivial zeros of $\zeta(s)$ lie close to $\Real(s)=1/2$. 
In this paper, we prove that this phenomenon also holds for our broad class $\mathcal{B}$ of analytic functions. To state our result, we set $F:= F(s,\bm{\pi})\in \mathcal{B}$ and define the following zero-counting functions:
\begin{align*}
    N_{F}^+(c; T_1, T_2) &:= \#\{ \text{nontrivial zeros of $F(s,\bm{\pi})$ with } T_1 < \Imag(s) < T_2 \text{ and } \Real(s) > c \},\\
    \intertext{and}
    N_{F}^-(c; T_1, T_2) &:= \#\{ \text{nontrivial zeros of $F(s,\bm{\pi})$ with } T_1 < \Imag(s) < T_2 \text{ and } \Real(s) < c \}.
\end{align*}

We now state our second main theorem, which shows that, under a strong second-moment assumption, almost all the nontrivial zeros of $F(s,\bm{\pi})$ are near the critical line.
\begin{thm}\label{thm:zeros-near-1/2}
    Let $F(s,\bm{\pi}) \in \mathcal{B}$ be defined as in \eqref{def:F} satisfying \eqref{assump:assumption 1}{and Hypothesis $\bm{\mathrm{V}}$}. Suppose each $L(s, \pi_u)$ satisfies 
     \begin{align}\int_T^{2T}\Big|L\Big(\frac{1}{2}+it,\pi_u \Big)\Big|^{2}dt \ll_{\eta,\pi_u} T(\log T)^{\eta}\label{eq:sharp second moment bound}
    \end{align}
for some $\eta > 0$. Then for any large $T$ and any $\delta > 0$, we have
\begin{equation}
N_{F}^{+}\Big(\frac{1}{2} + \delta; T, 2T\Big) + N_{F}^{-}\Big(\frac{1}{2} - \delta; T, 2T\Big) 
= O_F\Big(\frac{T \log \log T}{\delta}\Big),\label{eq:eq in theorem 1.3}
\end{equation}
where the implied constant on the right side of \eqref{eq:eq in theorem 1.3} depends only on $F$.
\end{thm}
\begin{rem}
  Theorem~\ref{thm:zeros-near-1/2} yields a nontrivial result when $\delta \gg_F {\log\log T}/{\log T}$. In this range, our theorem implies that almost all nontrivial zeros of $F(s,\bm{\pi})$ cluster near the critical line  $\Real(s)=1/2$.
\end{rem} 

Theorem \ref{thm:zeros-near-1/2} requires only the second-moment bound in \eqref{eq:sharp second moment bound}.
In many natural families of automorphic $L$-functions, this bound is known to hold under certain standard analytic hypotheses.
To make this connection explicit, we now introduce relevant sets of $L$-functions and formulate the auxiliary assumptions under which \eqref{eq:sharp second moment bound} follows automatically.
These will be used to deduce a more concrete corollary of Theorem \ref{thm:zeros-near-1/2}.

We introduce the following sets and hypotheses:
\begin{itemize}
    \item The set $\mathcal{R}$ consists of the $L$-functions corresponding to all representations in $\mathfrak{F}_1$ and a specific subset of  $\mathfrak{F}_2$. 
    This class consists of the $L$-functions of 
    holomorphic cusp forms of  $\Gamma_0(N)$ of even weight~$k$, and even Maa{\ss} cusp forms of $\Gamma_0(N)$, both for all $N \ge 1$.
    
    \item  The set $\mathcal{T}$ consists of all $L$-functions corresponding to the  representations in  $\cup_{i=1}^{4}\mathfrak{F}_i$.
\end{itemize}

\begin{hypoh}[Rudnick and Sarnak \cite{rudnick}]
        Let $m\in\N$ and $\pi\in\mathfrak{F}_m$, and let $\lambda_{\pi}(n)$ denote the 
        $n$-th Dirichlet coefficient of $L(s,\pi)$.
        For any fixed $k\ge 2$, 
    \begin{align*}
    \sum_{p\text{ prime}} \frac{|\lambda_\pi(p^k)|^2}{p^k} <\infty.
    \end{align*}
    \end{hypoh}
    
    Hypothesis $\mathrm{H}_{\pi}$ has been definitively established by Jiang \cite{Jiang} for all $m \geq 1$, although Jiang's proof is currently under review. Prior to Jiang’s work, Hypothesis $\mathrm{H}_{\pi}$ was known only in low dimensions: the case $m = 1$ is trivial; the case $m = 2$ is due to Kim and Sarnak~\cite{kim-Sarnak}; $m = 3$ is due to Rudnick and Sarnak~\cite{rudnick}; and $m = 4$ is due to Kim~\cite{kim}. Nevertheless, throughout this paper, we continue to state results under the explicit assumption of Hypothesis $\mathrm{H}_{\pi}$. This is done for structural clarity, although the assumption may now be omitted.

\begin{hypoA}
Let $\mathcal{S} \subset \mathfrak{S}$ be a finite set.
We say that $\mathcal{S}$ satisfies Hypothesis $\bm{\mathrm{A}}$ if every element $L(s,\pi) \in \mathcal{S}$ satisfies the following conditions:
\begin{enumerate}
    \item If $L(s,\pi) \in \mathcal{T} \setminus \mathcal{R}$, then the Riemann Hypothesis (RH) holds for $L(s,\pi)$;
    \item If $L(s,\pi) \in \mathcal{S} \setminus \mathcal{T}$, then both RH and Hypothesis $\bm{\mathrm{H}}_{\pi}$ hold for $L(s,\pi)$.
\end{enumerate}
\end{hypoA}
 
\begin{cor}\label{cor: corollary of thm 1.2}
Let $F(s,\bm{\pi}) \in \mathcal{B}$ be defined as in \eqref{def:F} satisfying \eqref{assump:assumption 1}{ and Hypothesis~$\bm{\mathrm{V}}$}. Suppose the set of $L$-functions appearing in the representation \eqref{def:F} of $F(s,\bm{\pi})$ satisfies Hypothesis $\bm{\mathrm{A}}$.
Then for any large $T$ and any $\delta > 0$, we have
\begin{equation}
N_{F}^{+}\Big(\frac{1}{2} + \delta; T, 2T\Big) + N_{F}^{-}\Big(\frac{1}{2} - \delta; T, 2T\Big) 
= O_F\Big(\frac{T \log \log T}{\delta}\Big),\label{eq:eq in theorem 1.3 2}
\end{equation}
where the implied constant on the right side of \eqref{eq:eq in theorem 1.3 2} depends only on $F$.
\end{cor}
\begin{rem}
    Throughout the paper, the notation $O_F(\cdot)$ means that the implied constant may depend on the fixed function $F(s,\bm{\pi})$. More explicitly, it may depend on the number of monomials $M$, the number of $L$-functions $V$, the differentiation orders for each $L$-function $k_u$, the coefficients $c_j$, the exponents $d_{u,l,j}$, and the underlying representations $\pi_u$ (including their dimensions, analytic conductors, spectral parameters, and bounds toward the Generalized Ramanujan--Petersson Conjecture (GRC)). All of these notations are defined in \eqref{def:F}.
\end{rem}

\subsection{Connections to earlier work and related problems}
   Our main results above recover and extend several classical results, and also establish connections to a range of existing problems. To illustrate the breadth of our framework, we consider several key instances. We first demonstrate how our theorems recover known results for specific choices of the function $F(s,\bm{\pi})$, and then show how they provide insights into related problems such as $a$-points and the zeros of truncated series.
\begin{enumerate}
[leftmargin=0.3in]
    \item \textbf{Derivatives of the Riemann zeta function $\zeta^{(k)}(s)$:} Taking $F(s, \bm{\pi}) =\zeta^{(k)}(s)$, the $k$-th derivative of the Riemann zeta function, we obtain from Theorem \ref{thm:number of zeros} 
    \[
    N_F(0,T) = \frac{T}{2\pi}\log\frac{T}{2\pi e} - \frac{T}{2\pi}\log 2 + O_F(\log T),
    \]
    which recovers the result of Berndt~\cite[Theorem]{Berndt}. Moreover, Theorem \ref{thm:zeros-near-1/2} implies that almost all of these zeros lie near the critical line, thereby recovering Theorem 2 in the work of Levinson and Montgomery~\cite{levinsonmontgomery}.
    \item \textbf{Polynomials in the derivatives of $\zeta(s)$:} Another example is when $F(s, \bm{\pi})$ is a polynomial in the derivatives of $\zeta(s)$, namely,
\[
    F(s, \boldsymbol{\pi}) = \sum_{j=1}^{M} c_j \prod_{l=0}^{k} \zeta^{(l)}(s)^{d_{l,j}},
\]
for some integer $M \ge 1$, integers $d_{l,j} \ge 0$, and nonzero complex constants $c_j$.
For such $F(s, \bm{\pi})$, applying Theorem \ref{thm:number of zeros} yields
\[
    N_F(0,T) = \deg_{\dim}(F)\frac{T}{2\pi}\log\frac{T}{2\pi e} - \frac{T}{2\pi}\log n_F + O_F(\log T).
\]
Additionally, Theorem \ref{thm:zeros-near-1/2} shows that almost all of these nontrivial zeros cluster near the critical line. These results recover work of Onozuka \cite[Theorems~1.3 and 1.5]{Onozuka}. {We note that Hypothesis $\bm{\mathrm{V}}$ (see Section \ref{subsec:notation}) holds automatically for every such $F(s,\bm\pi)$, since $V=1$ here; this is consistent with Onozuka's results requiring no analogous hypothesis.}

\item \textbf{$a$-points problems:} The study of the \textit{$a$-points} of $\zeta(s)$ concerns the distribution of the roots of $\zeta(s)-a$ for a fixed $a \in \mathbb{C}$. Originally investigated by Bohr, Landau, and Littlewood~\cite{bohr-Landau-Littlewood}, they provided the asymptotic count of roots of $\zeta(s)-a$ and showed that under RH these roots cluster near the critical line. Levinson \cite{levinson} later proved this clustering phenomenon unconditionally. On the other hand, Selberg \cite{selberg} conjectured that there are only finitely many roots of $\zeta(s)-a$ on the critical line. More recently, Lester \cite{lester} showed that for $a \neq 0$, asymptotically at most half of these points lie on the critical line. Beyond the $a$-points of $\zeta(s)$, Onozuka~\cite{onozuka-a-point} studied related $a$-point problems for the derivatives of the zeta function, $\zeta^{(k)}(s)$.

Our work can be applied to the study of $a$-points for $F(s,\bm{\pi})$, a polynomial in the derivatives of various automorphic $L$-functions. 
The connection is immediate: the $a$-points of $F(s,\bm{\pi})$ are precisely the zeros of the translated function $F(s,\bm{\pi})-a$. Applying our main theorems to this function yields two key consequences. First, Theorem \ref{thm:number of zeros} provides an asymptotic formula for the number of $a$-points, thereby recovering, as special cases, the classical results of Bohr, Landau, and Littlewood \cite{bohr-Landau-Littlewood} and Onozuka \cite[Theorem~1.1]{onozuka-a-point}. Second, Theorem \ref{thm:zeros-near-1/2} shows that, under its assumptions, almost all of these $a$-points lie close to the critical line $\Real(s) = 1/2$.  

As another illustration, our framework applies to the study of the $a\text{-points}$ of $({\zeta'}/{\zeta})(s)$, an important function in analytic number theory. Indeed, such $a$-points correspond precisely to the zeros of $\zeta'(s) - a \zeta(s)$, which is a linear combination of $\zeta(s)$ and $\zeta'(s)$. If all zeros of $\zeta(s)$ are assumed to be simple, we obtain a bijection between the $a$-points of $({\zeta'}/{\zeta})(s)$ and the zeros of $\zeta'(s) - a \zeta(s)$. Consequently, our theorems can be applied to study the distribution of these $a$-points  (given $a\neq 1$) of the logarithmic derivative.

 \item  \textbf{Zeros of truncated Taylor expansions:} Another interesting aspect of the present work is the following. Consider the Taylor expansion of $L(s+c,\pi)$ around the point $s$:
\[
L(s+c,\pi) = L(s,\pi) + L^{(1)}(s,\pi)c + \frac{L^{(2)}(s,\pi)}{2!}c^2 + \frac{L^{(3)}(s,\pi)}{3!}c^3 + \cdots.
\]
Under GRH, the nontrivial zeros of the function $L(s+c,\pi)$ lie on the shifted critical line $\Real(s) = 1/2 - \Real(c)$. However, any finite truncation of the series above is a linear combination of derivatives of $L(s, \pi)$. As such, Theorem \ref{thm:zeros-near-1/2} applies to these partial sums. This leads to a striking contrast: while the zeros of the full series are on a shifted line, our theorem implies that almost all zeros of its truncated polynomial approximations lie arbitrarily close to the original critical line, $\Real(s) = 1/2$. This reveals a stark dichotomy between the behavior of the above infinite series and its finite approximations. 

\end{enumerate}

\subsection*{Declarations}
A.D., N.W., and A.Z.  have no competing interest and did not receive any specific funding for this work. 

\subsection*{Author Contributions}
A.D., N.W., and A.Z. conceived the idea, developed the theory, and performed calculations. All authors contributed to the final manuscript and approved the version.

\subsection*{Outline of the paper}
The paper is organized as follows.
Section~\ref{sec:L-functions} reviews the necessary background on properties of automorphic $L$-functions.
Section~\ref{sec:asymp-F} establishes an asymptotic formula for $F(1-s,\widetilde{\bm{\pi}})$, serving as an analogue of the functional equation for complete $L$-functions.
Section~\ref{sec:ZFR} applies this formula to locate the zeros of $F(s,\bm{\pi})$, proving large zero-free regions and counting the number of trivial zeros lying in predictable disks.
Section~\ref{sec:asymp-zero} combines the results from previous sections to prove Theorem~\ref{thm:number of zeros}, an asymptotic formula for the number of nontrivial zeros of $F(s,\bm{\pi})$ up to height~$T$. 
Section~\ref{sec:zeros near critical line} presents the proof of Theorem~\ref{thm:zeros-near-1/2}, which shows that, under certain conditions, almost all nontrivial zeros of $F(s,\bm{\pi})$ lie close to the critical line. 

\section{Properties of automorphic \texorpdfstring{$L$}{L}-functions}
\label{sec:L-functions}

    %%%%%%%%%%%%%%%%%%%%%%%%%%%%%%%%%%%%%%%%%

    For $u=1,2,\cdots,V$, let $m_u \ge 1$ be an integer, $\mathbb{A}_{\mathbb{Q}}$ the ring of ad\`{e}les over $\mathbb{Q}$, and $\mathfrak{F}_{m_u}$ the set of all cuspidal automorphic representations of $\mathrm{GL}_{m_u}(\mathbb{A}_{\mathbb{Q}})$ whose central characters are unitary and normalized to be trivial on the diagonally embedded positive reals. For any $\pi_u \in\mathfrak{F}_{m_u}$, there exists a smooth admissible representation $\pi_{u,v}$ of $\GL_{m_u}(\mathbb{Q}_v)$ at each place $v$ of $\mathbb{Q}$ such that $\pi_u$ decomposes as the restricted tensor product $\otimes_v \pi_{u,v}$. 
    
     Given $\pi_u \in\mathfrak{F}_{m_u}$, we denote by $\mathfrak{q}_{\pi_u}$ its arithmetic conductor. At a non-archimedean \mbox{place $v$} corresponding to a prime $p$, the local $L$-function $L(s,\pi_{u,p})$ is defined through the Satake parameters $A_{\pi}(p)=\{\alpha_{1,\pi_u}(p),\ldots,\alpha_{m_u,\pi_u}(p)\}$ as
    \begin{equation}
	\label{eqn:Euler_p_single}
        L(s,\pi_{u,p})=\prod_{r=1}^{m_u}(1-\alpha_{r,\pi_u}(p)p^{-s})^{-1}=\sum_{k=0}^{\infty}\frac{\lambda_{\pi_u}(p^k)}{p^{ks}},
    \end{equation}
    where the coefficients $\lambda_{\pi_u}(p^k)$ are given by
    \[
    \lambda_{\pi_u}(p^k) = \sum_{\substack{\ell_1+\ell_2+\cdots+\ell_{m_u}=k\\\ell_i\in\mathbb{N}\cup\{0\}}}\alpha_{1,{\pi_u}}(p)^{\ell_1}\alpha_{2,{\pi_u}}(p)^{\ell_2}\cdots\alpha_{m_u,{\pi_u}}(p)^{\ell_{m_u}}.
    \]
    If $p\nmid \mathfrak{q}_{\pi_u}$, then for all $1\le r\le m_u$, we have $\alpha_{r,\pi_u}(p)\neq0$.  If $p\mid \mathfrak{q}_{\pi_u}$, then there might exist $r$ such that $\alpha_{r,\pi_u}(p)=0$.  The standard $L$-function $L(s,\pi_u)$ associated to $\pi_u$ is
    \begin{equation}\label{def:dirichlet series for Lu(s)}
        L(s,\pi_u)=\prod_{p} L(s,\pi_{u,p})=\sum_{n=1}^{\infty}\frac{\lambda_{\pi_u}(n)}{n^s}.
    \end{equation}
    The Euler product and Dirichlet series converge absolutely when $\mathrm{Re}(s)>1$. 
	
    At the archimedean place of $\mathbb{Q}$, the local $L$-function $ L(s,\pi_{u, \infty})$ is defined using $m_u$ spectral parameters $\mu_{\pi_u}(r)\in\mathbb{C}$:
    \[
        L(s,\pi_{u,\infty}) = \pi^{-m_us/2}\prod_{r=1}^{m_u}\Gamma\Big(\frac{s+\mu_{\pi_u}(r)}{2}\Big).
    \]

    Let $\widetilde{\pi}_u\in\mathfrak{F}_{m_u}$ be the contragredient representation of $\pi_u$. We have $\mathfrak{q}_{\pi_u}=\mathfrak{q}_{\widetilde{\pi}_u}$, and 
    \[
        \{\alpha_{r,\widetilde{\pi}_u}(p)\}=\{\overline{ \alpha_{r,\pi_u}(p)}\},\qquad \{\mu_{\widetilde{\pi}_u}(r)\}=\{\overline{\mu_{\pi_u}(r)}\}.
    \]
    Let $\omega_{\pi_u}$ be the order of the pole of $L(s,\pi_u)$ at $s=1$: this is $0$ unless $m_u = 1$ and $\pi_u$ is trivial, in which case $L(s,\pi_u)$ is the Riemann zeta function and has a simple pole at $s=1$.
    The completed $L$-function
    \[
        \Lambda(s,\pi_u) = (s(1-s))^{\omega_{\pi_u}}\mathfrak{q}_{\pi_u}^{s/2}L(s,\pi_u)L(s,\pi_{u,\infty})
    \]
    is entire of order $1$, and there exists a complex number $W(\pi_u)$ of modulus $1$ such that for all $s\in\mathbb{C}$, the functional equation $\Lambda(s,\pi_u)=W(\pi_u)\Lambda(1-s,\widetilde{\pi}_u)$ holds. The analytic conductor of $\pi_u$ \cite{IS} is given by
    \begin{equation}
	\label{eqn:analytic_conductor_def for Lu(s)}
            \mathfrak{C}(s, \pi_u)\coloneqq \mathfrak{q}_{\pi_u}\prod_{r=1}^{m_u} (3+|s+\mu_{\pi_u}(r)|),\qquad \mathfrak{C}(\pi_u)\coloneqq \mathfrak{C}(0,\pi_u).
    \end{equation}

    The Grand Riemann Hypothesis (GRH) states that all zeros of the completed $L$ function $\Lambda(s,\pi_u)$ are on the critical line $\Real(s) = 1/2$. We always indicate the $L$-functions for which we assume GRH. The Generalized Ramanujan--Petersson Conjecture (GRC) states that $\alpha_{r,\pi_u}(p)$ in \eqref{eqn:Euler_p_single} satisfy $|\alpha_{r,\pi_u}(p)|=1$ for all primes $p\nmid\mathfrak{q}_{\pi_u}$ and for all $1 \le r \le m_u$. In general, this conjecture is open. Towards GRC, Luo, Rudnick, and Sarnak \cite{LRS} and M\"uller and Speh \cite{MS} established the existence of a \mbox{constant $\theta_{m_u}\in[0,\frac{1}{2}-\frac{1}{m_u^2+1}]$} such that we have
    \begin{equation}
	\label{eqn:LRS_finite}
            |\alpha_{r,\pi_u}(p)|\le  p^{\theta_{m_u}}\qquad\text{ and }\qquad\Real(\mu_{\pi_u}(r))\ge -\theta_{m_u},
    \end{equation}
    and GRC predicts that one may take $\theta_{m_u}=0$ in \eqref{eqn:LRS_finite}. For the case $m_u=2$, Kim and Sarnak \cite[Appendix 2]{Kim-GL4} and Blomer and Brumley \cite{BB} proved the bound $\theta_2 \le 7/64$.

\section{An asymptotic functional equation for \texorpdfstring{$F(s,\bm{\pi})$}{F}}\label{sec:asymp-F}

In this section, we present technical lemmas underpinning the proof of Theorem \ref{thm:number of zeros}, beginning with a bound on the $n$-th coefficient of the Dirichlet series for $F(s,\bm{\pi})$.
\begin{lem}\label{lem:bound of coeff} 
Let $F(s,\bm{\pi}) \in \mathcal{B}$ be defined as in \eqref{def:F}, satisfying \eqref{assump:assumption 1} and Hypothesis $\bm{\mathrm{V}}$. Let $\theta_{m_u} \in [0, \frac{1}{2}-\frac{1}{m_u^2+1}]$ be the best bound towards GRC of $L(s, \pi_u)$ and let $\theta = \max\{\theta_{m_u}: 1 \le u \le V\}$. Recall the definition of $\eta_n$ in \eqref{eq:Dirichlet series}.
For any small $\varepsilon > 0$, we have $\eta_n = O_{F,\varepsilon}(n^{\theta+\varepsilon})$.
\end{lem}
\begin{proof}
For any $1 \le j \le M$ and $\Real(s) > 1$, the $j$-th term of $F(s,\bm{\pi})$ admits the expansion
\begin{align*}
    \sum_{n=1}^{\infty} \frac{\eta_n^{(j)}}{n^s} &:= c_j \prod_{u=1}^V \prod_{l=0}^{k_u} L^{(l)}(s,\pi_u)^{d_{u,l,j}} = c_j \prod_{u=1}^V\prod_{l=0}^{k_u}\Big(\sum_{n=1}^{\infty}  \frac{\lambda_u(n)(-\log n)^l}{n^s}\Big)^{d_{u,l,j}}.
\end{align*}
The coefficient $\eta_n^{(j)} \in \mathbb{C}$ can be described by the convolution of the coefficients in the series. For $1\le u\le V$, $0\le l \le k_u$, $0\le w \le d_{u,l,j}$, and $\alpha_{u;l;w} \in \mathbb{N}$, we have
\begin{align*}
    \eta_n^{(j)} &= c_j \sum_{\prod_{u=1}^V\prod_{l=0}^{k_u} \prod_{w=0}^{d_{u,l,j}}\alpha_{u;l ;w}
    = n } \Big[\prod_{u=1}^V\prod_{l=0}^{k_u} \prod_{w=0}^{d_{u,l,j}}\lambda_{u}(\alpha_{u;l ;w})  (-\log \alpha_{u;l;w})^l \Big].
\end{align*}
 Using \eqref{eqn:LRS_finite} and the definition of $\theta$, we have $\theta < 1/2$  and $|\lambda_{u}(\ell)| \le \ell^{\theta}\tau_m(\ell)$ uniformly for all $u$. Here,  $\tau_m(\ell)$ denotes the number of ways to write $\ell$ as the product of $m$ positive integers. Bounding each logarithmic term and each $\tau_m(\ell)$ trivially, we have for $\varepsilon > 0$,
\begin{equation*}
   \eta_n^{(j)}  \ll_{F,\varepsilon} n^{\theta+\varepsilon} (\log n)^{\sum_{u=1}^V\sum_{l=0}^{k_u} l d_{u,l,j}} \sum_{\prod_{u=1}^V\prod_{l=0}^{k_u} \prod_{w=0}^{d_{u,l,j}}\alpha_{u;l ;w}=n} 1 \ll_{F,\varepsilon} n^{\theta+\varepsilon}\tau_{{\sum_{u=1}^V\sum_{l=0}^{k_u} d_{u,l,j}}}(n),
\end{equation*}
which is $O_{F,\varepsilon}(n^{\theta+\varepsilon})$.
As a result, we have $\eta_n$ is at most $O_{F,\varepsilon}(Mn^{\theta+\varepsilon}) = O_{F,\varepsilon}(n^{\theta+\varepsilon})$.
\end{proof}

Next, we use the coefficient bound obtained from Lemma \ref{lem:bound of coeff} to derive an asymptotic formula \mbox{for $F(s,\bm{\pi})$} that is valid for all large $\Real(s)$.
\begin{lem}\label{lem:tail of series}
Write  $s = \sigma+it$. Let $F(s,\bm{\pi}) \in \mathcal{B}$ be defined as in \eqref{def:F}, satisfying \eqref{assump:assumption 1} and Hypothesis $\bm{\mathrm{V}}$. Recall the definition of $\eta_n$ and $n_F$ in \eqref{eq:Dirichlet series}. For large $\sigma$, we have
\begin{equation}\label{eq:tail of series}
    F(s,\bm{\pi}) = \frac{\eta_{n_F}}{n_F^s} + O_F((n_F+1)^{-\sigma}).
\end{equation}
\end{lem}
\begin{proof} Using \eqref{eq:Dirichlet series}, we can write
\begin{equation}\label{eq:dirichlet series-expansion}
    F(s,\bm{\pi}) = \frac{\eta_{n_F}}{n_F^s} + \frac{\eta_{n_F+1}}{(n_F+1)^s} + \sum_{n > n_F+1} \frac{\eta_n}{n^s}.
\end{equation}
The second term is clearly $O_F((n_F+1)^{-\sigma}).$  For the last term, applying Lemma \ref{lem:bound of coeff}, we obtain
\begin{align}
\sum_{n > n_F+1} \frac{\eta_n}{n^s} \ll_{F,\varepsilon} \sum_{n > n_F+1} \frac{n^{\theta+\varepsilon}}{n^\sigma} \ll_{F,\varepsilon} (n_F +1)^{1+\theta-\sigma+\varepsilon} \ll_{F} (n_F +1)^{-\sigma}.\label{eq:delicate upper bound for lemma 3.2}
\end{align}
The lemma follows by putting the estimates for the second and last terms in \eqref{eq:dirichlet series-expansion}.
\end{proof}

Next, for any $\bm{\pi} = ( \pi_1,\pi_2,\cdots,\pi_V)$, we define $\widetilde{\bm{\pi}} =(\widetilde{\pi}_1, \ldots, \widetilde{\pi}_V)$ and hence
\begin{equation}\label{def:F functional equation key component}
  F(s,\widetilde{\bm{\pi}}) = \sum_{j=1}^{M} c_j  \prod_{u=1}^V \prod_{l=0}^{k_u}  L^{(l)}(s,\widetilde{\pi}_u)^{d_{u,l,j}}.
\end{equation}
We then apply Lemmas~\ref{lem:bound of coeff} and~\ref{lem:tail of series} to establish the relation between  ${F}(1 - s,\widetilde{\bm{\pi}})$ and $\{L(s, \pi_u)\}_{1 \le u\le V}$. This relation serves as an analogue of the functional equation that relates $\zeta(s)$ to $\zeta(1-s)$ and is a key ingredient for the proof of Theorem \ref{thm:number of zeros}. Before proceeding, we define the auxiliary quantity $B(s,l,\pi_u)$ for $1\le u\le V$ and $0 \le l \le k_u$. If $l \neq 0$, we define
\begin{align}\label{def:B(k,pi
)}
B(s,l,\pi_u) &:= \hspace{-0.1in}\sum_{\substack{\sum_{1\le r \le m_u} (\ell_{4,r} + \ell_{5,r}) =l\\ 0 \le \ell_{4,r},\ell_{5,r} \le l }} \binom{l}{(\ell_{4,r})_{r=1}^{m_u}, (\ell_{5,r})_{r=1}^{m_u}} \notag\\
&\qquad \cdot \prod_{r=1}^{m_u} \frac{1}{2^{\ell_{4,r} + \ell_{5,r}}}\Big(\log \Big(\frac{s+\mu_{\pi_u}(r)}{2}\Big)\Big)^{\ell_{4,r}} \Big(\log \Big(\frac{1+s-\overline{\mu_{\pi_u}(r)}}{2}\Big)\Big)^{\ell_{5,r}}.
\end{align} 
Otherwise, if $l=0$, we define $B(s,0,\pi_u):=1$.

\begin{lem}\label{lem:FE}
Write $s= \sigma+it$. Let $F(s,\bm{\pi}) \in \mathcal{B}$ be defined as in \eqref{def:F} satisfying \eqref{assump:assumption 1}{and Hypothesis $\bm{\mathrm{V}}$}. If $c > 3/2$ and $\varepsilon > 0$, then in the region 
\[
R := \{ s \in \mathbb{C} : \sigma > c \} \cap \bigcup_{u=1}^{V}\bigcup_{r=1}^{m_u}\bigcup_{n=-\infty}^{\infty}\{s \in \mathbb{C} :  |s-(2n-1+\overline{\mu_{\pi_u}(r)})| \ge \varepsilon\},
\]
{for large $|t|$ ($\sigma$ confined to a bounded range of $\mathbb{R}$), there exists some $\Cr{conductor-dependence-1} := \Cr{conductor-dependence-1}(F) > 0$ and a fixed reference index $j_0 \in J$ such that, writing $e_{u,0}:=e_{u,j_0}=\sum_{l=0}^{k_u}d_{u,l,j_0}$, $F(1-s, \widetilde{\bm{\pi}})$ equals
\begin{equation}\label{eq:FE-lemma-2}
	(-1)^{\deg_{\mathrm{der}}(F(s,\bm{\pi}))}\Big(\sum_{j \in J} c_j\Big)\prod_{u=1}^V\Big(L(1-s,\widetilde{\pi}_u)^{e_{u,0}} \prod_{l=0}^{k_u} B(s,l,\pi_u)^{d_{u,l,j_0}} \Big)(1+O_{F}(e^{-\Cr{conductor-dependence-1}\sigma})),
\end{equation}
and, for large $\sigma$ ($t$ confined to a bounded range of $\mathbb{R}$), there exists a fixed reference index $j_0' \in J'$ such that, writing $e_{u,0}':=e_{u,j_0'}$, $F(1-s, \widetilde{\bm{\pi}})$ equals
\begin{equation}\label{eq:FE-lemma}
  (-1)^{\deg'_{\mathrm{der}}(F(s,\bm{\pi}))}\Big(\sum_{j \in J'} c_j\Big) \prod_{u=1}^V\Big(L(1-s,\widetilde{\pi}_u)^{e_{u,0}'} \prod_{l=0}^{k_u} B(s,l,\pi_u)^{d_{u,l,j_0'}} \Big)(1+O_{F}(1/\log s)),
\end{equation}}
where $L(1-s, \widetilde{\pi}_u)$ can be expressed in terms of $L(s, \pi_u)$ as in \eqref{eq:FE for L(s,pi)}.
\end{lem}

\begin{proof} 
Using the functional equation from Section \ref{sec:L-functions} and the definition of $L_{\infty}(s,\pi)$, we have for each $1\le u\le V$,
    \begin{multline}\label{eq:FE for L(s,pi)}
        L(1-s, \widetilde{\pi}_u)   = W(\pi_u)^{-1}\mathfrak{q}_{\pi_u}^{s-\frac{1}{2}}\pi^{-\frac{m_u}{2}-m_us}  L(s, \pi_u) \\
        \times \prod_{r=1}^{m_u}  \cos\Big(\frac{s-\overline{\mu_{\pi_u}(r)}}{2}\pi \Big) \Gamma\Big(\frac{s+\mu_{\pi_u}(r)}{2}\Big)\Gamma\Big(\frac{1+s-\overline{\mu_{\pi_u}(r)}}{2}\Big),
    \end{multline}
upon using the fact that $\Gamma(1-s)\Gamma(s) = \pi/{\sin(\pi s)}$. For any fixed $u$, we write
\begin{equation}\label{eq:factor}
    L(1-s,\widetilde{\pi}_u) = f_1(s)f_2(s)\prod_{r=1}^{m_u}f_{3,r}(s)f_{4,r}(s)f_{5,r}(s),
\end{equation}
where
\begin{alignat*}{2}
   &f_1(s) = W(\pi_u)^{-1} \mathfrak{q}_{\pi_u}^{s-\frac{1}{2}}\pi^{-\frac{m_u}{2}-m_us}, \\  &f_2(s) = L(s, \pi_u),
\end{alignat*}
and, for any $1 \le r \le m_u$,
\begin{alignat*}{2}    
    & f_{3,r} (s) = \cos\Big(\frac{s-\overline{\mu_{\pi_u}(r)}}{2}\pi\Big), \\
    &f_{4,r}(s) = \Gamma\Big(\frac{s+\mu_{\pi_u}(r)}{2}\Big),\\
    &f_{5,r}(s) = \Gamma\Big(\frac{1+s-\overline{\mu_{\pi_u}(r)}} {2}\Big). 
\end{alignat*}

Furthermore, for any $1\le r \le m_u$, let $\ell_1$, $\ell_2$, $\ell_{3,r}$, $\ell_{4,r}$, and $\ell_{5,r}$  be any integers in the range $[0,l]$ where $0\le l \le k_u$ such that 
\[
\ell_1+\ell_2+\sum_{r=1}^{m_u} (\ell_{3,r}+\ell_{4,r}+\ell_{5,r})=l.
\]
The corresponding binomial coefficient of the above partition is
\begin{equation}\label{eq:binom}
\binom{l}{\ell_1,\ell_2, (\ell_{3,r})_{r=1}^{m_u},(\ell_{4,r})_{r=1}^{m_u}, (\ell_{5,r})_{r=1}^{m_u}} =O_{l}(1).
\end{equation}
Differentiating \eqref{eq:factor} $l$ times and dividing both sides by $L(1-s,\widetilde{\pi}_u)$, we obtain
\begin{multline}\label{eq:diff-k-times}
\frac{L^{(l)}(1-s,\widetilde{\pi}_u)}{L(1-s,\widetilde{\pi}_u)} = (-1)^{l} \sum_{\substack{ \ell_1+\ell_2\\ +\sum_{r=1}^{m_u}(\ell_{3,r}+\ell_{4,r}+\ell_{5,r})=l}} \binom{l}{\ell_1,\ell_2, (\ell_{3,r})_{r=1}^{m_u},(\ell_{4,r})_{r=1}^{m_u}, (\ell_{5,r})_{r=1}^{m_u}}\\ 
\times \frac{f_1^{(\ell_1)}(s)}{f_1(s)} \frac{f_2^{(\ell_2)}(s)}{f_2(s)}\prod_{r=1}^{m_u}\frac{f_{3,r}^{(\ell_{3,r})}(s)}{f_{3,r}(s)}\frac{f_{4,r}^{(\ell_{4,r})}(s)}{f_{4,r}(s)}\frac{f_{5,r}^{(\ell_{5,r})}(s)}{f_{5,r}(s)}.
\end{multline}

We now compute each factor in the summand of \eqref{eq:diff-k-times}.
\begin{enumerate}[leftmargin=0.3in]
    \item \textbf{Contribution from $f_1(s)$}: By direct computation, we have 
    \[
    f_1^{(\ell_1)}(s) = \Big(\log\Big(\frac{\mathfrak{q}_{\pi_u}}{\pi^{m_u}}\Big)\Big)^{\ell_1} W(\pi_u)^{-1} \mathfrak{q}_{\pi_u}^{s-\frac{1}{2}}\pi^{-\frac{m_u}{2}-m_us},
    \]
    and hence 
    \begin{equation}\label{eq:f_1}
    \frac{f_1^{(\ell_1)}(s)}{f_1(s)} = \Big(\log\Big(\frac{\mathfrak{q}_{\pi_u}}{\pi^{m_u}}\Big)\Big)^{\ell_1}  \ll_{\ell_1, \pi_u} 1.
    \end{equation}

    \item \textbf{Contribution from $f_2(s)$:} Since we assume that $\sigma > c > 3/2$, we have 
    \[
    |L^{(\ell_2)}(s, \pi_u)| = 
    \Big|\Big(\sum_{n=1}^{\infty} \frac{\lambda_{u}(n)}{n^s}\Big)^{(\ell_2)}\Big| \le \sum_{n=1}^{\infty} \Big|\frac{\lambda_{u}(n)(\log n)^{\ell_2}}{n^s}\Big| \ll_{\ell_2,\varepsilon} \sum_{n=1}^{\infty} \frac{n^{\theta+\varepsilon}}{n^c} \ll_{\ell_2,\varepsilon} 1.
    \]
    On the other hand, since $L(s,\pi_u)$ converges absolutely for $\sigma>c>3/2$, using its Euler product, we also have $|1/{L(s, \pi_u)}|\ll 1$.
    Therefore, we conclude that 
    \[
    \frac{f_2^{(\ell_2)}(s)}{f_2(s)} =\frac{L^{(\ell_2)}(s, \pi_u)}{L(s, \pi_u)} \ll_{\ell_2} 1. 
    \]

    \item \textbf{Contribution from $f_{3,r}(s)$ for $1 \le r \le m_u$:} By direct computation, we have
    \begin{equation*}
        \begin{aligned}
           \frac{f_{3,r}^{(\ell_{3,r})}(s)}{f_{3,r}(s)} 
           &= \begin{cases} (-1)^{\frac{\ell_{3,r}}{2}} (\frac{\pi}{2})^{\ell_{3,r}} & \text{if } \ell_{3,r} \text{ is even}; \\ (-1)^{\frac{\ell_{3,r}-1}{2}} (\frac{\pi}{2})^{\ell_{3,r}} \tan(\frac{s-\overline{\mu_{\pi_u}(r)}}{2}\pi)& \text{if } \ell_{3,r} \text{ is odd}. \end{cases} 
        \end{aligned}
    \end{equation*}
    In the even case, it is clear that the contribution is $O_{\ell_{3,r}}(1)$, so it remains to consider the odd case.
    Setting $w = e^{i(s-\overline{\mu_{\pi_u}(r)})\frac{\pi}{2}}$, if $w \neq 0$, we can write
    \begin{equation}\label{eq:tangent}
        \Big|\tan\Big(\frac{s-\overline{\mu_{\pi_u}(r)}}{2}\pi\Big)\Big| = \Big|\frac{w-\frac{1}{w}}{w+\frac{1}{w}}\Big| = \Big|\frac{w^2-1}{w^2+1}\Big|.
    \end{equation}
    Otherwise, if $|w|$ approaches $0$, the right side of  \eqref{eq:tangent} approaches $1$, which is bounded.

    We first consider when $|w| \ge 2$. By the triangle inequality, we obtain 
    \[
    \Big|\tan\Big(\frac{s-\overline{\mu_{\pi_u}(r)}}{2}\pi\Big)\Big|  \le  \frac{|w|^2+1}{|w|^2-1} \le  \frac{|w|^2+\frac{|w|^2}{4}}{|w|^2-\frac{|w|^2}{4}} \le \frac{5}{3}.
    \]
    Suppose now that $0 < |w| \le 2$ and consider
    \[
    S_{1,u} = \{ 0< |w| \le 2 \} \cap \{ |s-(2n-1+\overline{\mu_{\pi_u}(r)})| \ge \varepsilon,(n\in\mathbb{Z})\}.
    \]
    Since $|s-(2n-1+\overline{\mu_{\pi_u}(r))}| \ge \varepsilon$ for all $n \in \mathbb{Z}$. then there exists $\varepsilon' > 0$ depending on $\varepsilon$ such that $w$ is bounded away from $\pm i$ by at least $\varepsilon'$, i.e., $|w \pm i| > \varepsilon'$. 
    We then have 
    \[
    S_{1,u} \subseteq  \{ |w| \le 2 \} \cap \{ |w \pm i| > \varepsilon' \} := S_{2,u}.
    \]
    Since $S_{2,u}$ is compact and $\frac{|w|^2+1}{|w|^2-1}$ is continuous on $S_{2,u}$, it follows that $\frac{|w|^2+1}{|w|^2-1}$ is bounded \mbox{in $S_{2,u}$.} Combining all the cases, we conclude that ${f_{3,r}^{(\ell_{3,r})}(s)}/{f_{3,r}(s)}$ is $O_{\ell_{3,r}}(1)$.

    \item \textbf{Contribution from $f_{4,r}(s)$ for $1 \le r \le m_u$:} 
    From \cite[p.~679]{spira}, we have
\[
    \frac{f_{4,r}^{(\ell_{4,r})}(s)}{f_{4,r}(s)} = \frac{\Gamma^{(\ell_{4,r})}\big((s+\mu_{\pi_u}(r))/{2}\big)}{\Gamma\big((s+\mu_{\pi_u}(r))/{2}\big)}
    = \frac{1}{2^{\ell_{4,r}}}\Big(\log \Big(\frac{s+\mu_{\pi_u}(r)}{2}\Big)\Big) ^{\ell_{4,r}}(1+O_{\ell_{4,r}, \pi_u}(1/\log s))).\]
    
    \item \textbf{Contribution from $f_{5,r}(s)$ for $1 \le r \le m_u$:} Similarly to $f_{4,r}(s)$, we have 
    \[
    \frac{f_{5,r}^{(\ell_{5,r})}(s)}{f_{5,r}(s)} = \frac{1}{2^{\ell_{5,r}}}\Big(\log \Big(\frac{1+s-\overline{\mu_{\pi_u}(r)}}{2}\Big)\Big) ^{\ell_{5,r}}(1+O_{\ell_{5,r},\pi_u}(1/\log s)).
    \]
\end{enumerate}

Going back to \eqref{eq:diff-k-times}, we extract the summands when
$\sum_{r= 1}^{m_u} (\ell_{4,r} + \ell_{5,r}) = l$ and bound the rest of the summands using the above approximations, that is, 
\begin{align}\label{eq:combine-factors}
L^{(l)}(1-s,\widetilde{\pi}_u) &= (-1)^{l} \sum_{\sum_{r=1}^{m_u} (\ell_{4,r} + \ell_{5,r}) =l}  \binom{l}{(\ell_{4,r})_{r=1}^{m_u}, (\ell_{5,r})_{r=1}^{m_u}}\prod_{r=1}^{m_u} \frac{f_{4,r}^{(\ell_{4,r})}(s)}{f_{4,r}(s)}\frac{f_{5,r}^{(\ell_{5,r})}(s)}{f_{5,r}(s)} L(1-s,\widetilde{\pi}_u)\notag\\
&\quad+ O_{l}\Big(\sum_{\substack{\sum_{r=1}^{m_u} (\ell_{4,r} + \ell_{5,r})<l}}  f_1^{(\ell_1)}(s)f_2^{(\ell_2)}(s) \prod_{r=1}^{m_u}f^{(\ell_{3,r})}_{3,r}(s)f^{(\ell_{4,r})}_{4,r}(s)f^{(\ell_{5,r})}_{5,r}(s)\Big),
\end{align}
where the second sum is taken over the condition $\ell_1+\ell_2+\sum_{r=1}^{m_u}(\ell_{3,r}+\ell_{4,r}+\ell_{5,r})=l$.

Recall the definition of $B(s,l, \pi_u)$ in \eqref{def:B(k,pi
)}. The main term in \eqref{eq:combine-factors} then equals
\begin{align}
     (-1)^{l}L(1-s,\widetilde{\pi}_u) B(s, l,\pi_u)(1+O_{\pi_u,l}(1/\log s)). \label{eq:main term}
\end{align}

On the other hand, the second summand on the right side of \eqref{eq:combine-factors} is bounded above by
\begin{equation*}
\begin{aligned}
    & O_{l} \Big(\sum_{\sum_{r=1}^{m_u} (\ell_{4,r} + \ell_{5,r}) < l} \hspace{-0.2in}L(1-s,\widetilde{\pi}_u)  \cdot \frac{f_1^{(\ell_1)}(s)}{f_1(s)} \frac{f_2^{(\ell_2)}(s)}{f_2(s)}\prod_{r=1}^{m_u}\frac{f_{3,r}^{(\ell_{3,r})}(s)}{f_{3,r}(s)}\frac{f_{4,r}^{(\ell_{4,r})}(s)}{f_{4,r}(s)}\frac{f_{5,r}^{(\ell_{5,r})}(s)}{f_{5,r}(s)}\Big)\\
    &= O_{l} \Big( \sum_{\sum_{r=1}^{m_u} (\ell_{4,r} + \ell_{5,r}) < l} \hspace{-0.2in} L(1-s,\widetilde{\pi}_u) \prod_{r=1}^{m_u} \Big(\log \Big(\frac{s+\mu_{\pi_u}(r)}{2}\Big)\Big) ^{\ell_{4,r}} \Big(\log \Big(\frac{1+s-\overline{\mu_{\pi_u}(r)}}{2}\Big)\Big) ^{\ell_{5,r}} \Big),
\end{aligned}
\end{equation*}
where both sums are taken over the condition $\ell_1+\ell_2+\sum_{r=1}^{m_u}(\ell_{3,r}+\ell_{4,r}+\ell_{5,r})=l$. Since $\sum_{r=1}^{m_u} (\ell_{4,r} + \ell_{5,r}) < l$, the second summand on the right side of \eqref{eq:combine-factors} can be absorbed in the error term in \eqref{eq:main term}, that is,
\begin{equation}\label{eq:FE-k-derivative}
    L^{(l)}(1-s,\widetilde{\pi}_u) = (-1)^{l}L(1-s,\widetilde{\pi}_u) B(s, l,\pi_u) (1+O_{\pi_u,l}({1}/{\log s})).
\end{equation}

Denote the $j$-th term of $F(1-s, \widetilde{\bm{\pi}})$ by
\[
F_j(1-s, \widetilde{\bm{\pi}}) := c_j  \prod_{u=1}^V \prod_{l=0}^{k_u}  L^{(l)}(1-s, \widetilde{\pi}_u)^{d_{u,l,j}}.
\]
We obtain from \eqref{eq:FE-k-derivative} that for any $1 \le j \le M$, $F_j(1-s, \widetilde{\bm{\pi}})$ equals
\begin{equation}\label{eq:each-monomial}
\begin{aligned}
    c_j(-1)^{\deg_{\der}(F_j(1-s, \bm{\pi}))}\prod_{u=1}^V\Big(L(1-s,\widetilde{\pi}_u)^{\sum_{l =0}^{k_u}d_{u,l, j}} \prod_{l=0}^{k_u} B(s,l,\pi_u)^{d_{u,l,j}} \Big)(1+O_{F_j}(1/\log s)).
\end{aligned}
\end{equation}
To finish the proof, we show that the asymptotic behavior of $F(1-s, \widetilde{\bm{\pi}})$ for large $|t|$ is determined entirely by the terms $F_j(1-s, \widetilde{\bm{\pi}})$ such that $j \in J$  and the asymptotic behavior of $F(1-s, \widetilde{\bm{\pi}})$ for large $\sigma$ is determined entirely by the terms $F_j(1-s, \widetilde{\bm{\pi}})$ such that $j \in J'$. That is, the contribution of any term with $j \not\in J$ (resp. $j \not\in J'$) can be absorbed into the error term in \eqref{eq:FE-lemma-2} (resp. \eqref{eq:FE-lemma}).

Recall the functional equation of $L(1-s, \widetilde{\pi}_u)$ in \eqref{eq:FE for L(s,pi)}.
We apply Stirling's formula,  
$|\Gamma(s)| \asymp  |s|^{\text{Re}(s) - 1/2} e^{-|\Imag(s)\arg(s)| - \Real(s)}$,  and the fact that  $|\cos(s)| \asymp e^{|\Imag(s)|}$, both of which hold uniformly for $s \in R$, to the right side of \eqref{eq:FE for L(s,pi)} and obtain
\begin{align}
|L(1-s, \widetilde{\pi}_u)| \asymp_{\pi_u} 
\mathfrak{q}_{\pi_u}^{\sigma} \Big(\Big|\frac{s}{2\pi e}\Big|^{\sigma-1/2}e^{|t|(\frac{\pi}{2}-|\arg(s)|)}\Big)^{m_u},%(1+O_{\pi_u}(e^{-\pi|t|}+{1}/{s})),
\label{eq:|L(1-s,pi u)| absolute asymp}
\end{align}
where $\arg(s) \in (-\pi/2, \pi/2)$. On the other hand, since 
\[
\sum_{\sum_{1\le r \le m_u} (\ell_{4,r} + \ell_{5,r}) =l} \binom{l}{(\ell_{4,r})_{r=1}^{m_u}, (\ell_{5,r})_{r=1}^{m_u}} \prod_{r=1}^{m_u} \frac{1}{2^{\ell_{4,r}+\ell_{5,r}}}  = m_u^l, 
\]
it follows that  
\begin{align}
    \prod_{l=0}^{k_u} B(s,l,\pi_u)^{d_{u,l,j}} =  
    (m_u \log s)^{\sum_{l=0}^{k_u}l d_{u,l, j}}(1+O_{\pi_u}(1/\log s)).\label{eq:asymp B}
\end{align}
Applying \eqref{eq:|L(1-s,pi u)| absolute asymp} and \eqref{eq:asymp B}, we have the main term in \eqref{eq:each-monomial} is 
\begin{equation*}
 \asymp_{F_j}  \prod_{u=1}^{V}  \Big[\mathfrak{q}_{\pi_u}^{\sigma} \Big(\Big|\frac{s}{2\pi e}\Big|^{\sigma-1/2}e^{|t|(\frac{\pi}{2}-|\arg(s)|)}\Big)^{m_u}\Big]^{\sum_{l =0}^{k_u}d_{u,l, j}} (\log s)^{\sum_{l=0}^{k_u}l d_{u,l, j}},
\end{equation*}
which can be simplified to 
\begin{equation}\label{eq:comparison}
e^{\sigma \deg_{\cond}(F_j(1-s, \bm{\pi}))}  \Big(\Big|\frac{s}{2\pi e}\Big|^{\sigma-1/2}e^{|t|(\frac{\pi}{2}-|\arg(s)|)}\Big)^{\deg_{\dim}(F_j(1-s, \bm{\pi}))} (\log s)^{\deg_{\der}(F_j(1-s, \bm{\pi}))}. %\prod_{u=1}^V m_u^{\sum_{l=0}^{k_u}ld_{u,l,j}}.
\end{equation}

Thus, for the $j$-th term $F_j(1-s,\widetilde{\bm{\pi}})$, we reach the following conclusion regarding its contribution relative to the dominant terms ($j \in J$ or $j \in J'$). For notational simplicity, we use the abbreviations $F := F(1-s,\widetilde{\bm{\pi}})$ and $F_j := F_j(1-s,\widetilde{\bm{\pi}})$. From \eqref{eq:comparison}, it follows that
\begin{enumerate}[leftmargin=0.3in]
    \item \textbf{Dimension-weighted degree:} For large $|s|$, if $\deg_{\dim}(F_j) < \deg_{\dim}(F)$, the contribution of $F_j$ is smaller by a factor of at least $O(1/\log s)$ relative to the maximal dimension-weighted terms. This holds regardless of the magnitudes of $\deg_{\der}(F_j)$, $\deg_{\cond}(F_j)$, $\deg'_{\der}(F_j)$, or $\deg'_{\cond}(F_j)$.

    \item \textbf{For large $|t|$:}  If $\deg_{\dim}(F_j) = \deg_{\dim}(F)$ but $\deg_{\der}(F_j) < \deg_{\der}(F)$, the contribution of this $F_j$ is smaller by a factor of at least $O(1/\log s)$, relative to those terms that achieve both $\deg_{\dim}(F)$ and $\deg_{\der}(F)$.

   If  $\deg_{\dim}(F_j) = \deg_{\dim}(F)$ and $\deg_{\der}(F_j) = \deg_{\der}(F)$ but $\deg_{\cond}(F_j) < \deg_{\cond}(F)$,  then the contribution of $F_j$ is smaller by a factor of at least $O(e^{-\Cr{conductor-dependence-1}\sigma})$ for some $\Cl[abcon]{conductor-dependence-1} := \Cr{conductor-dependence-1}(F) > 0$, compared to the terms that achieve all $\deg_{\dim}(F)$, $\deg_{\der}(F)$, and $\deg_{\cond}(F)$.

    \item \textbf{For large $\sigma$:} If $\deg_{\dim}(F_j) = \deg_{\dim}(F)$ but $\deg_{\cond}(F_j) < \deg'_{\cond}(F)$, the contribution of this $F_j$ is smaller by a factor of at least $O(e^{-\Cr{conductor-dependence-2}\sigma})$ for some $\Cl[abcon]{conductor-dependence-2} := \Cr{conductor-dependence-2}(F) > 0$,  relative to those terms that achieve both $\deg_{\dim}(F)$ and $\deg'_{\cond}(F)$. 
    
   If  $\deg_{\dim}(F_j) = \deg_{\dim}(F)$ and $\deg_{\cond}(F_j) = \deg'_{\cond}(F)$ but $\deg_{\der}(F_j) < \deg'_{\der}(F)$,  then the contribution of $F_j$ is smaller by a factor of at least $O(1/ \log s)$, compared to the terms that simultaneously achieve $\deg_{\dim}(F)$, $\deg'_{\der}(F)$, and $\deg'_{\cond}(F)$.
\end{enumerate}

{The bullets above compare each excluded term $F_j$, $j\notin J$
(resp.\ $j\notin J'$), against a single fixed reference term
$F_{j_0}$ with $j_0\in J$ (resp.\ $F_{j_0'}$ with $j_0'\in J'$).
Summing over the finitely many excluded indices $j$, we obtain $F(1-s,\widetilde{\bm{\pi}}) - \sum_{j\in J}F_j = O_F(F_{j_0}(1/(\log s)+e^{-\Cr{conductor-dependence-1}\sigma}))$.
Indeed, the terms excluded because of a gap in
$\deg_{\mathrm{dim}}$ or $\deg_{\mathrm{der}}$ contribute
$O_F(F_{j_0}/\log s)$, while those excluded only because of a gap in
$\deg_{\mathrm{cond}}$ contribute
$O_F(F_{j_0}e^{-\kappa_1\sigma})$.
Similarly, for the large-$\sigma$ asymptotic indexed by $J'$, the
excluded terms contribute $O_F(F_{j_0'}/\log s)$.

It remains to identify $\sum_{j\in J}F_j$ (resp.\ $\sum_{j\in J'}F_j$) with a single explicit function of $s$, and to check it is not far smaller than $F_{j_0}$, both need Hypothesis $\bm{\mathrm{V}}$. Since every $j\in J$ shares the same type vector, \eqref{eq:each-monomial} together with \eqref{eq:asymp B} shows that for every $j\in J$,
\[
\prod_{u=1}^V\Big(L(1-s,\widetilde{\pi}_u)^{e_{u,j}}\prod_{l=0}^{k_u}B(s,l,\pi_u)^{d_{u,l,j}}\Big) = \prod_{u=1}^V\Big(L(1-s,\widetilde{\pi}_u)^{e_{u,0}}\prod_{l=0}^{k_u}B(s,l,\pi_u)^{d_{u,l,j_0}}\Big)\big(1+O_F(1/\log s)\big),
\]
both sides sharing the same leading asymptotic $\prod_u L(1-s,\widetilde\pi_u)^{e_{u,0}}(m_u\log s)^{\ell_{u,0}}$ regardless of how the derivative orders are split within each $j$. Summing over $j\in J$ and bringing out by the nonzero constant $\sum_{j\in J}c_j$ (condition \eqref{assump:assumption 1}) gives
\[
\sum_{j \in J} F_j = \Big(\sum_{j\in J} c_j\Big)(-1)^{\deg_{\mathrm{der}}(F)}\prod_{u=1}^V\Big(L(1-s,\widetilde{\pi}_u)^{e_{u,0}}\prod_{l=0}^{k_u}B(s,l,\pi_u)^{d_{u,l,j_0}}\Big)(1+O_F(1/\log s)),
\]
which is $\asymp_F |F_{j_0}|$ as required. Combining this with the estimate for the excluded terms, we have $F(1-s,\widetilde{\bm{\pi}})$ equals
\[
\Big(\sum_{j\in J}c_j\Big)
(-1)^{\deg_{\mathrm{der}}(F)}
\prod_{u=1}^V
\Big(
L(1-s,\widetilde{\pi}_u)^{e_{u,0}}
\prod_{l=0}^{k_u}B(s,l,\pi_u)^{d_{u,l,j_0}}
\Big)
\left(
1+
O_F\left(e^{-\kappa_1\sigma}\right)
\right),
\]
which proves \eqref{eq:FE-lemma-2}. For the large-$\sigma$ asymptotic,
the same argument with $J'$ in place of $J$, together with the
$O_F(1/\log s)$ bound for all excluded terms from the large-$\sigma$
part of the preceding comparison, gives \eqref{eq:FE-lemma}.} 
\end{proof}

{
\begin{rem}[The degree invariants are intrinsic]\label{rem:intrinsic}
Although the representation \eqref{def:F} need not be unique, the degree
invariants defined above are intrinsic to $F(s,\bm{\pi})$. Indeed, taking
absolute values in \eqref{eq:FE-lemma} and using \eqref{eq:comparison}
together with Hypothesis $\bm{\mathrm{V}}$, as $\sigma\to\infty$ we obtain
\begin{align*}
    \log |F(1-\sigma,\widetilde{\bm{\pi}})|
    =
&\deg_{\dim}(F)\,\sigma\log\sigma +
\big(
\deg_{\mathrm{cond}}'(F)
-\deg_{\dim}(F)\log(2\pi e)
\big)\sigma\\
&-\frac{1}{2}\deg_{\dim}(F)\log\sigma
+\deg_{\mathrm{der}}'(F)\log\log\sigma
+O_F(1).
\end{align*}
Since the left-hand side depends only on $F$, the coefficients of the
successive growth terms $\sigma\log\sigma$, $\sigma$, and $\log\log\sigma$
uniquely determine $\deg_{\dim}(F)$, $\deg_{\mathrm{cond}}'(F)$, and $\deg_{\mathrm{der}}'(F)$.
Similarly, the large-$|t|$ asymptotic in \eqref{eq:FE-lemma-2} uniquely
determines $\deg_{\mathrm{der}}(F)$ from the modulus, while
$\deg_{\mathrm{cond}}(F)$ is determined by the linear term in the
continuously tracked argument along a fixed vertical line, as in
\eqref{eq:arg(F(E1+iT,pi))}. Thus all five degree invariants depend only
on $F(s,\bm{\pi})$ (with $\bm{\pi}$ fixed), and not on the choice of
representation \eqref{def:F}. %In particular, any two representations satisfying \eqref{assump:assumption 1} and Hypothesis $\bm{\mathrm{V}}$ yield the same values of these invariants, and hence the same constants $\alpha_1$ and $\alpha_2$ in Theorem \ref{thm:number of zeros}.
\end{rem}
}

\section{Zero-free regions and the location of trivial zeros}\label{sec:ZFR}
In this section, we apply the  results from the previous section to study the zeros of $F(s,\bm{\pi})$. We begin by locating the \textit{trivial} zeros of $F(s,\bm{\pi})$ and establishing a zero-free region, which confines the remaining zeros of $F(s,\bm{\pi})$ to a vertical strip. The existence of such a strip has been studied for special cases of our work. In the case of the derivatives of $\zeta(s)$, Spira \cite{spira} established a zero-free region $\Real(s) \ge 1.75k + 2$ for $k \ge 3$, and later \cite{spira-2} proved the existence of a corresponding left half-plane, $\Real(s) \le \alpha_k$ for some real $\alpha_k$, containing only trivial real zeros.
The following proposition proves the existence of such a strip in our general case.
\begin{prop}\label{prop:zfr1}
    Write $s= \sigma+it$. Let $F(s,\bm{\pi}) \in \mathcal{B}$ be defined as in \eqref{def:F} satisfying \eqref{assump:assumption 1}{and Hypothesis $\bm{\mathrm{V}}$}. For $\varepsilon > 0$, there exist constants $E_{1}<0$ and $E_2 > 1$, depending only on $F$ and $\varepsilon$, such that $F(s,\bm{\pi}) \neq 0$ in the region
\[
\Bigr(\{ s \in \mathbb{C} : \sigma \le E_1 \} \cap \bigcup_{u=1}^{V}\bigcup_{r=1}^{m_u}\bigcup_{n=-\infty}^{\infty}\{ s \in \mathbb{C} : |s+2n+\mu_{\pi_u}(r)| \ge \varepsilon \}\Bigr) \cup \{ s \in \mathbb{C} :  \sigma \ge E_2 \} .
\]
\end{prop}
\begin{proof}
    We first prove the existence of $E_1$. 
{By \eqref{eq:FE-lemma} and \eqref{eq:asymp B} (applied to  $j_0'\in J'$ from Lemma \ref{lem:FE}), $F(1-s, \widetilde{\bm{\pi}})$ equals
     \begin{equation}\label{eq:FE-lemma-horizon}
    	(-\log s)^{\deg'_{\mathrm{der}}(F(s,\bm{\pi}))}\Big(\sum_{j \in J'} c_j\Big) \Big(\prod_{u=1}^V  m_u^{\ell_{u,0}'}\Big) \Big(\prod_{u=1}^VL(1-s,\widetilde{\pi}_u)^{e_{u,0}'} \Big)(1+O_{F}(\tfrac{1}{\log s})),
    \end{equation}
    where $e_{u,0}':=e_{u,j_0'}$, $\ell_{u,0}':=\ell_{u,j_0'}$ are the common values from Lemma \ref{lem:FE}, already a single function of $s$.}
     We denote the main term of $F(1-s, \widetilde{\bm{\pi}})$ in \eqref{eq:FE-lemma-horizon} by $A_1(s, \widetilde{\bm{\pi}})$ and the error term by $A_2(s, \widetilde{\bm{\pi}})$: 
    \begin{align}
    F(1-s, \widetilde{\bm{\pi}}) = A_1(s, \widetilde{\bm{\pi}}) + A_2(s, \widetilde{\bm{\pi}}).\label{eq:F(1-s,widetilde pi) decomposition}
    \end{align}
    By \eqref{eq:FE-lemma}, there exists some constant $\Cl[abcon]{FE} := \Cr{FE}(F)>0$ such that 
    \begin{equation}\label{eq:A1-A2}
        |A_2(s, \widetilde{\bm{\pi}})| < \Cr{FE} {|A_1(s, \widetilde{\bm{\pi}})|}/{|\log s|}
    \end{equation}
    for any $s \in R$ (the region defined in   Lemma \ref{lem:FE}) with $c > 3/2$.

   	We now show that $A_1(s, \widetilde{\bm{\pi}}) \neq 0$ for all $s \in  R$ with a  sufficiently large $\sigma$. 
    From \eqref{eq:FE for L(s,pi)}, each factor $L(1-s,\widetilde{\pi}_u)$ is holomorphic and nonzero throughout $R$ (its only zeros for $\sigma > c$ occur at the points excluded from $R$), and $\prod_u m_u^{\ell_{u,0}'} > 0$ automatically. Our non-destruction assumption \eqref{assump:assumption 1} ensures $\sum_{j \in J'} c_j \ne 0$. Therefore $A_1(s, \widetilde{\bm{\pi}}) \neq 0$ for all $s \in R$ with a sufficiently large real part $\sigma$.

    Thus, if $\sigma$ is sufficiently large, we have from \eqref{eq:A1-A2} and the fact that $A_1(s, \widetilde{\bm{\pi}}) \neq 0$  that $|A_2(s,\widetilde{\bm{\pi}})| < |A_1(s,\widetilde{\bm{\pi}})|$. Then there exists $E'_1 := E'_1(F,\varepsilon)$ such that $F(1-s,\widetilde{\bm{\pi}})$ is nonzero in 
    \[
    \{s \in \mathbb{C} : \sigma \ge E'_1 \} \cap R. 
    \]
By change of variables $s \mapsto 1-\overline{s}$ and complex conjugation, 
if we set $E_1 := \min\{1-E'_1, 1-c\}$, we may conclude that $F(s, \bm{\pi}) \neq 0$ in the region
\[
\{s \in \mathbb{C} : \sigma \le E_1\} \cap \bigcup_{u=1}^{V}\bigcup_{r=1}^{m_u}\bigcup_{n=-\infty}^{\infty}\{s \in \mathbb{C} :  |s+2n+ \mu_{\pi_u}(r)| \ge \varepsilon \}.
\]

To prove the existence of $E_2$, we consider \eqref{eq:tail of series} and denote the main term of $F(s,\bm{\pi})$ by $B_1(s,\bm{\pi})$ and its error term by $B_2(s,\bm{\pi})$. In particular, for large $\sigma > 1$, we have
\[
F(s,\bm{\pi}) =  B_1(s,\bm{\pi}) + B_2(s,\bm{\pi}).
\]
Then there exists some constant $\Cl[abcon]{tail} > 0$ such that $|B_2(s,\bm{\pi})| <  \Cr{tail} (n_F+1)^{-\sigma}.$
Then for sufficiently large, that is, for 
\[
\sigma \ge \max\Big\{1, \frac{\log \Cr{tail}- \log|\eta_{n_F}|}{\log\big(1+{1}/{n_F}\big)}\Big\} := E_2,
\]
we have 
\[
|B_2(s,\bm{\pi})| < \Cr{tail} (n_F+1)^{-\sigma} \le {|\eta_{n_F}|}{n_F^{-\sigma}} = |B_1(s,\bm{\pi})|.
\]
Since $B_1(s,\bm{\pi}) \neq 0$ by the definition of $n_F$, it follows that $F(s,\bm{\pi})\neq 0$ in the region  $\{s \in \mathbb{C} : \sigma \ge E_2\}$. The proposition then follows by combining the two zero-free regions.
\end{proof}

\begin{lem}\label{thm:zfr2}
Fix $\varepsilon>0$, and let $E_1<0$ and $E_2>1$ be as in
Proposition \ref{prop:zfr1}. For any fixed $E_1'\leq E_1$ and
$E_2'\geq E_2$, $F(s,\bm{\pi})$ has only finitely many trivial zeros
with real part in $[E_1',E_2']$.
\end{lem}

\begin{proof}
By the definition of the trivial zeros, they lie in the neighborhoods
\[
\bigcup_{u=1}^{V}
\bigcup_{r=1}^{m_u}
\bigcup_{n\in\mathbb{Z}}
\left\{
s\in\mathbb{C}:
|s+2n+\mu_{\pi_u}(r)|<\varepsilon
\right\}.
\]
For a fixed pair $(u,r)$, such a disk can intersect the strip
$E_1'\leq\Real(s)\leq E_2'$ only if
\[
-2n-\Real(\mu_{\pi_u}(r))
\in
[E_1'-\varepsilon,E_2'+\varepsilon].
\]
Hence, for each $(u,r)$, only finitely many integers $n$ can occur.
Since there are only finitely many pairs $(u,r)$, the trivial zeros
of $F(s,\bm{\pi})$ in the strip are contained in a finite union of
bounded disks.

The closure of this finite union is compact. Since $F(s,\bm{\pi})$
is holomorphic on a neighborhood of this compact set except possibly
at $s=1$, which is a pole rather than a zero, and $F$ is not
identically zero, its zeros are isolated. Therefore this compact set
contains only finitely many zeros of $F(s,\bm{\pi})$, and hence only
finitely many trivial zeros.
\end{proof}

%%%%%%%%%%%%%%%%%%%%%%%%%%%%%%%%%
%%%%%%%%%%%%%%%%%%%%%%%%%%%%%%%%%
%%%%%%%%%%%%%%%%%%%%%%%%%%%%%%%%%

\section{An asymptotic zero-counting formula}\label{sec:asymp-zero}
This section applies the preceding results to complete the proof of Theorem \ref{thm:number of zeros}. We first establish more analytic properties of $F(s,\bm{\pi})$, showing that the function is of order 1, and, once completed to an entire function, admits a Hadamard product factorization.
\begin{lem}\label{lem:hadamard}
Let $F(s,\bm{\pi}) \in \mathcal{B}$ be defined as in \eqref{def:F}. Let $z_F$  be the order of zeros at $s=0$ of $F(s,\bm{\pi})$. Then there exists a complex number $B_F$ and an integer $p_F \ge 0$ such that 
\begin{align}\label{eq:log-derivative}
    \frac{F'(s,\bm{\pi})}{F(s,\bm{\pi})} = \frac{p_F}{1-s}+\frac{z_F}{s}+ B_F + \sum_{\rho_F \neq 0} \Big( \frac{1}{s-\rho_F} + \frac{1}{\rho_F} \Big),
\end{align}
where $\rho_F$ runs over all zeros of $F(s,\bm{\pi})$, counted with multiplicity.
\end{lem}
\begin{proof} 
Write $s=\sigma+it$. For each $1\le u \le V$, we first estimate the growth of  $L(s,\pi_u)$ for when $|t|$ is large and $\sigma$ is fixed in the following ranges:
\begin{enumerate}[leftmargin=0.3in]
    \item If $\sigma \ge 1$, then $L(\sigma +it, \pi_u) \ll 1$.
    \item Suppose that $\sigma \le 0$. By the functional equation of $L(s,\pi_u)$ and Stirling's formula, 
    \begin{equation}\label{eq:growth-L-<0}
   \begin{aligned}
       L(1-s, \pi_u)  
       &\ll_{\pi_u}  \mathfrak{q}_{\pi_u}^{\sigma-1/2} %\pi^{-m_u\sigma}  
       L(s, \widetilde{\pi}_u) |s|^{m_u\sigma-m_u/2}.
   \end{aligned} 
\end{equation}
    Recalling the definition of analytic conductor in \eqref{eqn:analytic_conductor_def for Lu(s)} and taking $s$ such \mbox{that $\Real(s)\ge1$} in \eqref{eq:growth-L-<0}, we have for any $s$ with $\Real(s) \le 0$ and any large $t$
    \[
    L(s, {\pi}_u) \ll_{\pi_u} \mathfrak{C}(s, \pi_u)^{1/2-\sigma}. 
    \]
    \item When $0 < \sigma < 1$, we use upper bounds for $L(-\varepsilon+it, \pi_u)$ and $L(1+\varepsilon+it, \pi_u)$ from the above approximations, then apply the Phragm\'en-Lindel\"of principle.  It follows that
\begin{equation} \label{eq:1.25}
L(s, \pi_u) \ll_{\pi_u,\varepsilon} \mathfrak{C}(s, \pi_u)^{(1-\sigma)/2+\varepsilon}.
\end{equation}
\end{enumerate}
In conclusion, if we define 
\begin{equation}\label{def:mu}
\Omega(\sigma) = 
\begin{cases}
    0 & (\sigma \ge 1); \\
    {1/2-\sigma/2} & (0 < \sigma < 1); \\
    {1/2-\sigma} & (\sigma \le 0),
\end{cases}
\end{equation}
then $L(s,\pi_u) \ll_{\pi_u,\varepsilon}\mathfrak{C}(s, \pi_u)^{\Omega(\sigma)+\varepsilon}$ for all $s \in \mathbb{C}$. By the definition of $\mathfrak{C}(s, \pi_u)$, the function $L(s,\pi_u)$ is a function of order 1.

We now verify that $F(s,\bm{\pi})$ is also a function of order $1$. 
It is immediate that if $f(s)$ and $g(s)$ are functions of order $1$, then both $f(s)+g(s)$ and $f(s)g(s)$ are of order $1$. Recall that $F(s,\bm{\pi})$ is a linear combination of derivatives of $L(s,\pi_u)$ where $1 \le u \le V$. Therefore, it is enough to show that all derivatives of $L(s,\pi_u)$ are functions of order $1$. In fact, by induction, it suffices to  prove that $L'(s,\pi_u)$ is of order $1$.

Let $C$ be the circle of radius 1 centered at $s$. For sufficiently large $t$, the function $L(z,\pi_u)$ is holomorphic for all $z$ inside and on $C$. 
Recall Cauchy's integral formula for derivatives,
\[
L'(s,\pi_u)  = \frac{1}{2\pi i} \int_C \frac{L(z,\pi_u)}{(z-s)^2} \, dz.
\]
For $s$ sufficiently large,
\begin{align}\label{eq:comparing}
|L'(s,\pi_u)| 
&= \left| \frac{1}{2\pi i} \int_C \frac{L(z,\pi_u)}{(z-s)^2} \, dz \right| 
\le \frac{1}{2\pi} \sup_{z \in C} |L(z,\pi_u)| \int_C \frac{|dz|}{|z-s|^2}  = \sup_{z \in C} |L(z,\pi_u)|.
\end{align}
This means $L'(s,\pi_u)$ has the order at most 1.

Therefore, $F(s,\boldsymbol{\pi})$ is holomorphic on $\mathbb{C}\setminus\{1\}$, with a possible pole at $s=1$ only if there is $\pi_u$ that is the trivial representation of $\mathfrak{F}_1$. Let $p_F$ be the order of this pole. It follows that the function $(s-1)^{p_F} F(s,\boldsymbol{\pi})$ is entire and of order 1, thus admitting a Hadamard product factorization. Specifically, there exist constants $A_F, B_F \in \mathbb{C}$ such that
\begin{equation}\label{eq:hadamard}
    (s-1)^{p_F} F(s,\boldsymbol{\pi}) = s^{z_F} e^{A_F + B_F s}
    \prod_{\rho_F \neq 0} \left( 1 - \frac{s}{\rho_F} \right) e^{s / \rho_F},
\end{equation}
where $z_F$ is the order of the zero of $F(s,\boldsymbol{\pi})$ at $s=0$, and the product runs over all nonzero zeros $\rho_F$, counted with multiplicity. The desired claim is then obtained by taking the logarithmic derivative of \eqref{eq:hadamard}.
\end{proof}

The following lemma provides a bound on the number of zeros in short vertical intervals, which is a key input for the final counting argument.
\begin{lem}\label{lem:local-zeros}
Let $F(s,\bm{\pi}) \in \mathcal{B}$ be defined as in \eqref{def:F}, satisfying \eqref{assump:assumption 1} and Hypothesis $\bm{\mathrm{V}}$. Then for any large $T$, 
\[
N_{F(s,\bm{\pi})}(T, T+1) \ll_F \log T.
\]
\end{lem}
\begin{proof}  We first establish lower and upper bounds for $F(s, \bm{\pi})$. Write $s = \sigma+it$. By Lemma \ref{lem:tail of series}, we have
\[
F(s,\bm{\pi}) - \frac{\eta_{n_F}}{n_F^s} = O_F\left((n_F+1)^{-\sigma}\right).
\]
From Proposition \ref{prop:zfr1}, we know that all the nontrivial zeros of $F(s, \bm{\pi})$ are confined to the strip $E_1 \le \sigma \le E_2$. Then there exists a constant $D \ge E_{2}$ such that 
\[
\Big|F(s,\bm{\pi}) - \frac{\eta_{n_F}}{n_F^s}\Big| \le \frac{1}{2}\Big|\frac{\eta_{n_F}}{n_F^s}\Big|
\]
holds for $\sigma \ge D$. By the triangle inequality, we have for $\sigma \ge D$
\begin{equation}\label{eq:F-triangle}
|F(s,\bm{\pi})| \ge \frac{1}{2}\Big|\frac{\eta_{n_F}}{n_F^s}\Big|.
\end{equation}

From the proof of Lemma \ref{lem:hadamard}, all derivatives of $L(s,\pi_u)$ can also be bounded above in the same manner as $L(s,\pi_u)$. In other words, we have for any integers $l \ge 0$, $L^{(l)}(s,\pi_u) \ll_{\pi_u,\varepsilon} \mathfrak{C}(s, \pi_u)^{\Omega(\sigma)+\varepsilon}$, where $\Omega(\sigma)$ is defined in \eqref{def:mu}. Consequently, for $|t| \to \infty$ and $\sigma$ confined to a vertical strip depending solely on $F$, we obtain
\begin{equation}\label{eq:F-growth}
    \log |F(s, \bm{\pi})| \ll_F \log |t|. 
\end{equation}

Now, let $n(r)$ be the number of zeros of $F(s,  \bm{\pi})$ counted with multiplicity in the circle with center $D+iT$ and radius $R$.
By Jensen's theorem, we have
\begin{align}\label{eq:Jensen}
   \int_0^{D-E_{1}+2} \frac{n(R)}{R} dR &= \frac{1}{2\pi} \int_0^{2\pi} \log\left|F\left(D+iT+(D-E_{1}+2)e^{i\theta}, \bm{\pi}\right)\right| d\theta\notag\\
   &\quad- \log|F(D+iT, \bm{\pi})|. 
\end{align}
From \eqref{eq:F-triangle}, we have  $|F(D+iT, \bm{\pi})| \ge \frac{1}{2}{|\eta_{n_F}|}/{n_F^D}$. Then there exists a constant $d_1$ depending on $F$ and $D$ such that
\begin{equation}\label{eq:lower-F(D)}
\log|F(D+iT, \bm{\pi})| \ge d_1.
\end{equation}
From \eqref{eq:F-growth}, there exists a constant $d_2 > 0$ depending on $F$ and $D$ such that
\begin{equation}\label{eq:upper-F(D)}
    \log\left|F\left(D+iT+(D-E_{1}+2)e^{i\theta}, \bm{\pi}\right)\right| \le d_2\log T. 
\end{equation}
Putting together \eqref{eq:Jensen},\eqref{eq:lower-F(D)},  \eqref{eq:upper-F(D)},  and the fact that $\int_0^{D-E_{1}+2} (n(R)/R)dR \ge 0$, we have
\begin{equation}\label{eq:Jensen1}
\int_0^{D-E_{1}+2} \frac{n(R)}{R} dR 
\ll_F \log T.
\end{equation}

Lastly, we observe that
\begin{equation}\label{eq:Jensen2}
\int_0^{D-E_{1}+2} \frac{n(R)}{R} dR \ge \int_{D-E_{1}+1}^{D-E_{1}+2} \frac{n(R)}{R} dR 
\gg_F n(D-E_{{1}}+1).
\end{equation}
Since $D \ge E_1$, the circle center at $D+iT$ with radius $D-E_1+1$ encompasses the rectangle with vertices at $E_1+iT, E_1+i(T+1), E_2+iT,$ and $E_2+i(T+1)$ (see Figure \ref{fig:local-zero}).  
Thus
\[
N_{F}(T, T+1) \le n(D-E_{1}+1), 
\]
which, together with \eqref{eq:Jensen1} and \eqref{eq:Jensen2}, completes this lemma.
\begin{figure}
    \centering
    \includegraphics[width=0.5\textwidth]{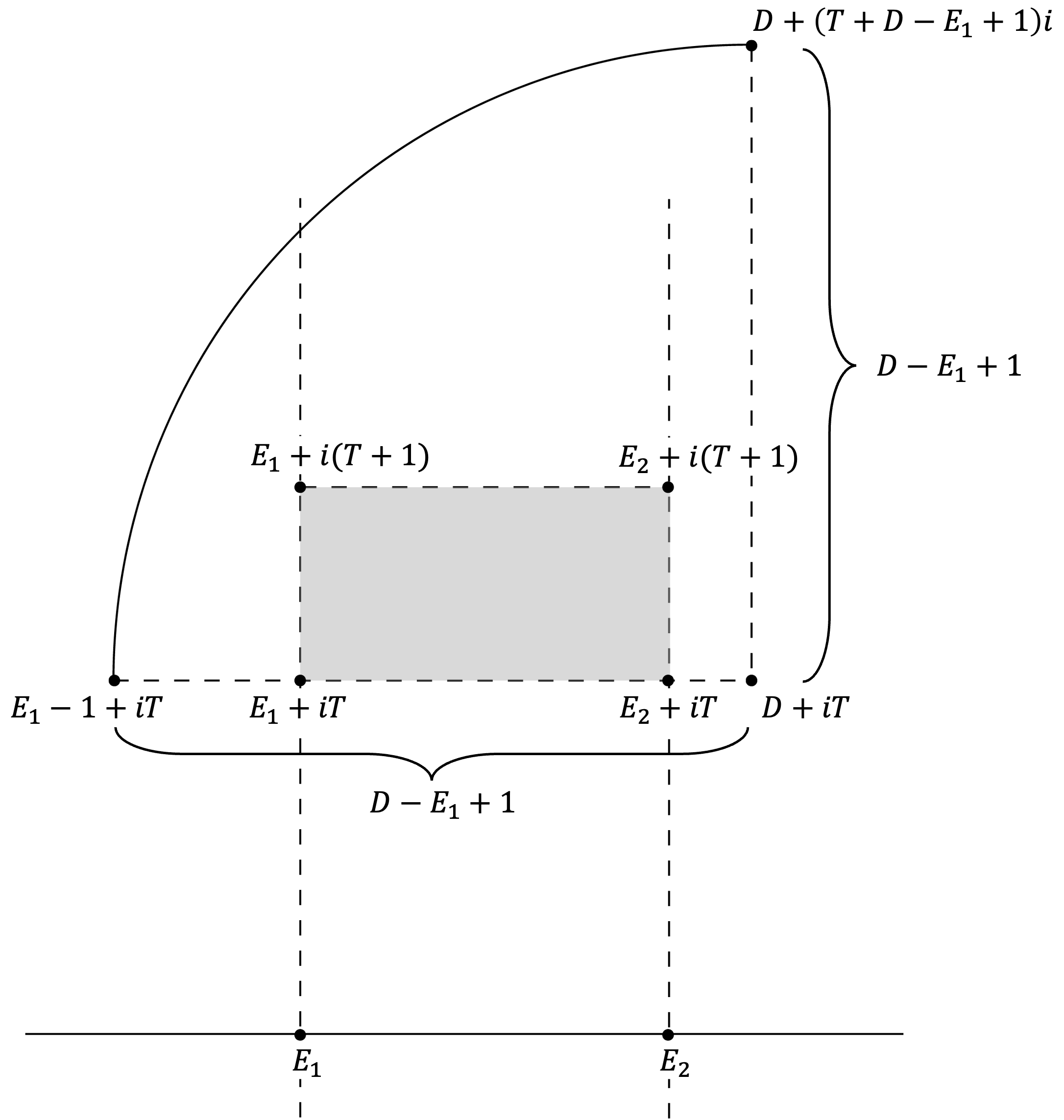}
    \caption{The rectangular contour and the encompassing quarter-circular arc}
    \label{fig:local-zero}
\end{figure}
\end{proof}

Finally, we derive an approximation for the logarithmic derivative of $F(s,\bm{\pi})$ in terms of its zeros near height $t$.
\begin{lem}\label{lem:log-derivative}
Let $F(s,\bm{\pi})\in \mathcal{B}$ be defined as in \eqref{def:F}. Let $\sigma_1 < \sigma_2$ be fixed real numbers and $\rho_F = \beta_F+i\gamma_F$ denote the zeros of $F(s,\bm{\pi})$. Then for any $s = \sigma+it$ such that $\sigma_1 < \sigma < \sigma_2$ and for any large $t$, we have
\[
\frac{F'(s,\bm{\pi})}{F(s,\bm{\pi})} = \sum_{|t - \gamma_F| < 1} \frac{1}{s-\rho_F} + O_F(\log t).
\]
\end{lem}
\begin{proof} 
For sufficiently large $E > E_2$ (as defined in Proposition \ref{prop:zfr1}), similarly to the proof of \eqref{eq:F-triangle}, we have 
\[
|F(E+it, \bm{\pi})| \ge \frac{1}{2} \frac{|\eta_{n_F}|}{n_F^E}.
\]
Similarly, we also have $|F(E+it,\bm{\pi})| \le 2{|\eta_{n_F}|}/{n_F^E}$ which, upon using \eqref{eq:comparing}, implies 
\[
|F'(E+it, \bm{\pi})| \le 2 \frac{|\eta_{n_F}|}{n_F^E}.
\]
Putting together, we have 
\begin{equation}\label{eq:F'/F-bounded}
    \frac{F'(E+it, \bm{\pi})}{F(E+it, \bm{\pi})} \ll_F 1.
\end{equation}

Now taking $s= E+it$ in \eqref{eq:log-derivative} from Lemma \ref{lem:hadamard}, we obtain
\begin{equation}\label{eq:log-derivative-E}
    \frac{F'(E+it, \bm{\pi})}{F(E+it, \bm{\pi})}  = \frac{p_F}{1-E-it}+\frac{z_F}{E+it} + B_F + \sum_{\rho_F \neq 0} \Big( \frac{1}{E+it-\rho_F} + \frac{1}{\rho_F} \Big).
\end{equation}
Subtracting \eqref{eq:log-derivative-E} from \eqref{eq:log-derivative} and using \eqref{eq:F'/F-bounded}, we obtain for $\sigma_1 < \sigma < \sigma_2$ and  large $t$,  
\begin{equation}\label{eq:log-derivative-difference}
\begin{aligned}
    \frac{F'(s, \bm{\pi})}{F(s, \bm{\pi})} 
    &= \Big( \sum_{|t-\gamma_F|<1} + \sum_{|t-\gamma_F|\ge 1} \Big) \left(\frac{1}{s-\rho_F} + \frac{1}{E+it-\rho_F}\right) + O_F(1).
\end{aligned}
\end{equation}
 For the terms with $|t-\gamma_F| \ge 1$, we have
\[
\sum_{|t-\gamma_F|\ge 1} \left| \frac{1}{s-\rho_F} - \frac{1}{E+it-\rho_F} \right| = \sum_{|t-\gamma_F|\ge 1} \frac{E-\sigma}{|(s-\rho_F)(E+it-\rho_F)|}\ll_F \sum_{|t-\gamma_F|\ge 1}\frac{1}{|t-\gamma_F|^2}.
\]
To estimate the quantity on the right side, we split the range of $\gamma_F$ into $\gamma_F > t+1$ and $\gamma_F < t-1$ and divide each range into intervals of length 1. We provide details for only when $\gamma_F > t+1$ as the other case can be done similarly. From Lemma \ref{lem:local-zeros}, we have
\begin{equation*}
    \begin{aligned}
        \sum_{\gamma_F \ge t+1} \frac{1}{|t-\gamma_F|^2} &= \sum_{n<t} \sum_{\gamma_F  =t+n}^{t+n+1}\frac{1}{|t-\gamma_F|^2} +\sum_{n > t}\sum_{\gamma_F  =t+n}^{t+n+1}\frac{1}{|t-\gamma_F|^2}\\
        &\ll_F \sum_{n<t} \frac{\log(t+n)}{(n+1)^2} +  \sum_{n>t} \frac{\log(t+n)}{(n+1)^2},
    \end{aligned}
\end{equation*}
which is $O_F(\log t)$. Therefore, we conclude that 
\begin{equation}\label{eq:outer-range-2}
\sum_{|t-\gamma_F|\ge 1} \left(\frac{1}{s-\rho_F} + \frac{1}{E+it-\rho_F}\right) \ll_F \log t.
\end{equation}
On the other hand, by Lemma \ref{lem:local-zeros}, we  have that 
\begin{equation}\label{eq:inner-range}
    \sum_{|t-\gamma_F|<1} \frac{1}{E+it-\rho_F} \ll_F \sum_{|t-\gamma_F|<1} 1 \ll_F \log t.
\end{equation}
The lemma then follows once we put together \eqref{eq:outer-range-2} and \eqref{eq:inner-range} in \eqref{eq:log-derivative-difference}.
\end{proof}

%%%%%%%%%%%%%%%%%%%%%%%%%%%%%%%%%%%%%%%%%%%%%%%%%%%
%%%%%%%%%%%%%%%%%%%%%%%%%%%%%%%%%%%%%%%%%%%%%%%%%%%
With all the preceding lemmas and propositions, we now prove our main result.
\subsection*{Proof of Theorem \ref{thm:number of zeros}}
Recall the definition of $E_1$ and $E_2$ from Proposition \ref{prop:zfr1}. Let $E'_{1} \le E_{1}$ and $E'_{2} \ge E_{2}$ be constants chosen appropriately small and large, respectively, as will be determined in the proof.
By Lemma \ref{thm:zfr2}, $F(s,\bm{\pi})$ has only finitely many trivial zeros with real parts in $[E'_1, E'_2]$ and imaginary parts in $[0, T]$. Moreover, Lemma \ref{lem:hadamard} implies that $F(s,\bm{\pi})$ has a pole of finite order at $s=1$. By Cauchy's theorem, we have
\begin{align*}
    N_{F(s, \bm{\pi})}(0,T) &= \frac{1}{2\pi}  \Imag\Big(\int_{E'_{1}}^{E'_{2}} + \int_{E'_{2}}^{E'_{2}+iT} + \int_{E'_{2}+iT}^{E'_{1}+iT} + \int_{E'_{1}+iT}^{E'_{1}}\Big)\frac{F'(s, \bm{\pi})}{F(s, \bm{\pi})} ds +O_F(1)\\
    &=: \frac{1}{2\pi} (J_1 + J_2 + J_3 + J_4) +O_F(1).
\end{align*}

The first integral, $J_1$, is along a horizontal line segment of fixed length, and its value is independent of $T$. Therefore, we have $J_1=O_F(1)$.

To estimate $J_2$, we substitute $\sigma = E'_2$ in Lemma \ref{lem:tail of series} and obtain
\begin{align*}
    \log F(s,\bm{\pi}) 
    = \log\left(\frac{\eta_{n_F}}{n_F^s}\right) + \log\left(1 + O_F\left(\left(\frac{n_F}{n_F+1}\right)^{\sigma}\right)\right) = \log\left(\frac{\eta_{n_F}}{n_F^s}\right) +  O_F(1).
\end{align*}
Thus the argument of $F(s,\bm{\pi})$ is
\[
\arg(F(s, \bm{\pi})) = \arg(\eta_{n_F}) -  t\log({n_F}) + O_F(1), 
\]
and hence 
\begin{align*}
    J_2 = [\arg(F(s, \bm{\pi}))]_{s=E'_{2}}^{s=E'_{2}+iT} = \left[ \arg(\eta_{n_F}) - t\log n_F + O_F(1) \right]_{t=0}^{t=T} 
    = -T\log n_F + O_F(1).
\end{align*}

Next, we estimate $J_3$. Applying the approximation from Lemma \ref{lem:log-derivative}, we have
\begin{equation*}
    J_3 = \Imag \Big(\sum_{|T-\gamma_F|<1} \int_{E'_{2}+iT}^{E'_1+iT} \frac{1}{s-\rho_F} ds\Big) + O_F\left(\log T \right).
\end{equation*}
For each integral, we change the path of integration. If $\gamma_F < T$, then we replace the path with the upper semicircle centered at $\rho_F$ and radius 1. If $\gamma_F > T$, then we replace the path with the lower semicircle centered at $\rho_F$ and radius 1. Then all integrals are bounded by $O(1)$. Therefore, by Lemma \ref{lem:local-zeros}, we have
\[
J_3 = \sum_{|\gamma_F-T|<1} O_F(1) + O_F\left(\log T\right) = O_F\left(\log T\right).
\]

Finally, we estimate $J_4$. Since the endpoint $\arg(F(E'_{1}, \bm{\pi}))$ is independent of $T$, we have
\[
J_4 = [\arg(F(s, \bm{\pi}))]_{E'_{1}+iT}^{E'_{1}} = -\arg(F(E'_{1}+iT, \bm{\pi})) + O_F(1).
\]

{By \eqref{eq:FE-lemma-2} of Lemma \ref{lem:FE}, with the fixed reference index $j_0 \in J$ and $e_{u,0}:=e_{u,j_0}$ from that lemma, $\arg(F(1-s,\bm{\pi}))$ equals
\begin{align}
 \arg\Big[\prod_{u=1}^V\Big(L(1-s,{\pi}_u)^{e_{u,0}} \prod_{l=0}^{k_u} B(s,l,\widetilde{\pi}_u)^{d_{u,l,j_0}} \Big)\Big]+O_F(1),\label{eq:argument-J4}
\end{align}
where the $O_F(1)$ absorbs both $\arg(1+O_{F}(e^{-\Cr{conductor-dependence-1}\sigma}))$ and the fixed, $T$-independent argument of the nonzero constant $\sum_{j\in J}c_j$ (nonzero by \eqref{assump:assumption 1}); no sum over $J$ remains, since \eqref{eq:FE-lemma-2} already collapsed it to a single function.

 For each $1 \le r \le m_u$, we write $\mu_{\widetilde{\pi}_u}(r)=a_{u,r}+ib_{u,r}$.  From \eqref{eq:FE for L(s,pi)} and Stirling's formula,  there exists a constant $Y_{\pi_u}$ depending only on $\widetilde{\pi}_u$ and $\sigma$ such that for large $|t|$
\begin{align}
	L(1-s,\pi_u)
	&=Y_{\pi_u} L(s,\widetilde{\pi}_u)\mathfrak{q}_{\pi_u}^{it}(2\pi e)^{-im_ut}e^{-\frac{m_u\pi t}{2}}|t|^{m_us- \frac{m_u}{2}}|t|^{-i\sum_{r=1}^{m_u}b_{u,r}}(1+O_{\pi_u}(1/|t|)).
\end{align}
Thus there exists a constant $Y_{j_0}$ depending only on $F_{j_0}$ and $\sigma$ such that
\begin{multline}\label{eq:contribution-L}
     \prod_{u=1}^VL(1-s,{\pi}_u)^{e_{u,0}} =Y_{j_0}\Big(\prod_{u=1}^V(L(s,\widetilde{\pi}_u)\mathfrak{q}_{\pi_u}^{it}|t|^{-i\sum_{r=1}^{m_u}b_{u,r}})^{e_{u,0}}\Big)\\
    \times ((2\pi e)^{-it}e^{-\frac{\pi t}{2}}|t|^{s-\frac{1}{2}})^{\deg_{\dim}(F(s,\bm{\pi}))}(1+O_F(1/|t|)).
\end{multline}
Furthermore, by \eqref{eq:asymp B}, there exists a constant $Y'_{j_0}$ depending only on $F_{j_0}$ such that
\begin{align}
     \prod_{u=1}^V \prod_{l=0}^{k_u} B(s,l,\widetilde{\pi}_u)^{d_{u,l,j_0}} 
    &=Y'_{j_0} (\log s)^{\deg_{\mathrm{der}}(F(s,\bm{\pi}))}(1+O_F(1/\log |t|)).\label{eq:the contribution of B(s,l,pi) in theorem 1.1}
\end{align}

Putting together \eqref{eq:contribution-L} and \eqref{eq:the contribution of B(s,l,pi) in theorem 1.1} in \eqref{eq:argument-J4}, we have  for large $T$
\begin{multline}
   \arg(F(E'_{1}+iT, \bm{\pi})) = \arg\Big(Y_{j_0}Y'_{j_0}\prod_{u=1}^V(L(1-E'_{1}-iT,\widetilde{\pi}_u)\mathfrak{q}_{\pi_u}^{-iT}T^{-i\sum_{r=1}^{m_u}b_{u,r}})^{e_{u,0}}\Big)\\
   +\deg_{\dim}(F(s,\bm{\pi}))\arg\Big((2\pi e)^{iT}e^{\frac{\pi T}{2}}T^{\frac{1}{2}-E'_{1}-iT}\Big)+\deg_{\mathrm{der}}(F(s,\bm{\pi}))\arg( \log (1-E'_{1}-iT))+O_F(1).\label{eq:arg(F(E1+iT,pi))}
\end{multline}
The second summand on the right side of \eqref{eq:arg(F(E1+iT,pi))} contributes 
\begin{align}\label{eq:second-term-J4}
   \deg_{\dim}(F(s,\bm{\pi}))(T\log(2\pi e) - T\log T)+O_F(1),
\end{align}
and the third summand in \eqref{eq:arg(F(E1+iT,pi))} contributes $O_F(1)$.  

    It remains to compute the contribution from the first summand in \eqref{eq:arg(F(E1+iT,pi))}. Since the $L$-function $L(s, \widetilde{\pi}_u)$ is in the region of absolute convergence for $\Real(s) = 1-E'_1$ (given that $E'_1$ is chosen sufficiently small, $E'_1 < E_1 <0$), its argument is  bounded by $O_{\pi_u}(1)$, and using
     \begin{align}
       &\prod_{u=1}^V\mathfrak{q}_{\pi_u}^{-iTe_{u,0}} = \exp(-iT\sum_{u=1}^V (\log \mathfrak{q}_{\pi_u})e_{u,0}) = \exp(-iT\deg_{\cond}(F(s,\bm{\pi}))),\label{eq:removing q pi in arg} 
     \end{align} 
     together with the standard bound $O_F(\log T)$ for the remaining phase factors $T^{-i\sum_{r} b_{u,r}}$, the first summand contributes
     \[
     -T\deg_{\mathrm{cond}}(F(s,\bm{\pi}))+O_F(\log T).
     \] }

Combining these estimates, the total asymptotic for $J_4$ is
\begin{align*}
J_4 &=  \deg_{\dim}(F(s,\bm{\pi}))(T\log T-T\log(2\pi e))+T\deg_{\mathrm{cond}}(F(s,\bm{\pi}))+O_F(\log T).
\end{align*}
Therefore, the final formula for $N_{F(s, \bm{\pi})}(0,T)$ using our estimates for $J_1, J_2, J_3,$ and $J_4$ becomes
\begin{align*}
    N_{F(s, \bm{\pi})}(0,T) &= \frac{1}{2\pi}(J_1+J_2+J_3+J_4) \\
    &= \deg_{\dim}(F(s,\bm{\pi}))\frac{T}{2\pi}\log \frac{T}{2\pi e}+\frac{T}{2\pi}\deg_{\mathrm{cond}}(F(s,\bm{\pi}))- \frac{T}{2\pi}\log n_F+ O_F(\log T), 
\end{align*}
which can then be rearranged into the desired form stated in the theorem.

%%%%%%%%%%%%%%%%%%%%%%%%%%%%%%%%%
%%%%%%%%%%%%%%%%%%%%%%%%%%%%%%%%%
%%%%%%%%%%%%%%%%%%%%%%%%%%%%%%%%%

\section{Zeros of \texorpdfstring{$F(s,\bm{\pi})$}{F(s,π)} near the critical line}\label{sec:zeros near critical line}
In this section, we prove Theorem \ref{thm:zeros-near-1/2}, which demonstrates that almost all nontrivial zeros of $F(s,\bm{\pi})$ are concentrated near the critical line. This extends the well-known phenomenon for the Riemann zeta function to our much more general setting.
\begin{lem}\label{lem:lemma 6.1}
    Let $F(s,\bm{\pi})\in \mathcal{B}$ be defined as in \eqref{def:F} satisfying \eqref{assump:assumption 1}{and Hypothesis $\bm{\mathrm{V}}$}, and let $\rho_F = \beta_F+i\gamma_F$ denote the nontrivial zeros of $F(s,\bm{\pi})$. For any large $T$, and a fixed real number $b<D$ ($D$ taken from the proof of Lemma \ref{lem:local-zeros}), we have
    \[
    2\pi\sum_{\substack{T<\gamma_F<2T\\\beta_F>b}}(\beta_F-b) = \int_T^{2T}\log |F(b+it,\bm{\pi})|dt-T\log\left\lvert\frac{\eta_{n_F}}{n_F^b}\right\rvert+O_F(\log T).
    \]
\end{lem}
\begin{proof}
    Write $s = \sigma +it$. For a real $b<D$, by applying Littlewood's lemma to 
\begin{equation}\label{def:G}
    G(s,\bm{\pi}) := \frac{F(s,\bm{\pi})}{\eta_{n_F}/n_F^s},
    \end{equation}
    we obtain \begin{multline}\label{eq:littlewood lemma application}
        2\pi\sum_{\substack{T<\gamma_F<2T\\\beta_F>b}}(\beta_F-b) = \int_T^{2T}\log |G(b+it,\bm{\pi})|dt-\int_T^{2T}\log |G(D+it,\bm{\pi})|dt\\
        +\int_b^D \arg G(\sigma+i2T,\bm{\pi}) d\sigma-\int_b^D \arg G(\sigma+iT,\bm{\pi})d\sigma.
    \end{multline}
    Note that by Lemma \ref{lem:tail of series}, $\lim\limits_{\sigma\rightarrow\infty}G(\sigma+it,\bm{\pi})=1$ for any fixed $t$, so in \eqref{eq:littlewood lemma application} we take the branch of $\arg(G(s,\bm{\pi}))$ such that $\arg(G(s,\bm{\pi}))\rightarrow 0$ as $\sigma\rightarrow \infty$. 
    
    We first show that the last two integrals in \eqref{eq:littlewood lemma application} can be absorbed into the error term $O_F(\log T)$. Define
    \[
    H_T(s,\bm{\pi}) := \frac{G(s+iT,\bm{\pi})+\overline{G(\overline{s}+iT,\bm{\pi})}}{2}.
    \]
    Note that $H_T(\sigma,\bm{\pi}) = \Real(G(\sigma+iT,\bm{\pi}))$. Since the argument can change by at most $\pi$ between consecutive zeros, if $H_T(\sigma,\bm{\pi})$ has $n$ zeros in $\sigma \in [b,D]$, then for any $\sigma \in [b,D]$ 
    \[
    |\arg G(\sigma+iT, \bm{\pi})|\le \pi n+O_F(1).
    \]
    Hence, it suffices to bound the number of zeros of $H_T(s,\bm{\pi})$ in the circle with center $D$ and radius $D-b$. Let $n(R)$ be the number of zeros of $H_T(s,\bm{\pi})$ counted with multiplicity in the circle with center $D$ and radius $R$. By Jensen's theorem, we have
   \begin{align}\label{eq:Jensen 3}
   \int_0^{D-b+1} \frac{n(R)}{R} dR &= \frac{1}{2\pi} \int_0^{2\pi} \log\left|H_T\left(D+(D-b+1)e^{i\theta},\bm{\pi}\right)\right| d\theta- \log|H_T(D,\bm{\pi})|. 
\end{align}
From \eqref{eq:F-growth} and the definition of $H_T(s,\bm{\pi})$, we obtain from \eqref{eq:Jensen 3} 
\begin{align}\label{eq:Jensen5}
    \int_0^{D-b+1} \frac{n(R)}{R} dR \ll_F \log T. 
\end{align}
Furthermore, we observe that
\begin{equation}\label{eq:Jensen4}
\int_0^{D-b+1} \frac{n(R)}{R} dR \ge \int_{D-b}^{D-b+1} \frac{n(R)}{R} dR 
\gg_F n(D-b). 
\end{equation}
Hence, combining \eqref{eq:Jensen5} and \eqref{eq:Jensen4}, we finally obtain that for any $\sigma\in [b,D]$
\begin{align*}
    \arg G(\sigma+iT,\bm{\pi}) \ll n(D-b)\ll_F \log T,
\end{align*}
which shows that the last two terms in \eqref{eq:littlewood lemma application} can be bounded by the error term in the lemma.

The remaining task is to estimate the second term in \eqref{eq:littlewood lemma application}. Notice that 
\begin{align*}
   \left|\int_T^{2T}\log  \left|G(D+it,\bm{\pi})\right|dt\right| = \left|\int_T^{2T} \Real(\log G(D+it,\bm{\pi}))dt\right| \le  \left|\int_T^{2T}\log G(D+it,\bm{\pi})dt\right|.
\end{align*}
For our choice of $D$, we have $\Real(G(s,\bm{\pi}))\ge 1/2$ for any $\sigma \ge D$. Therefore, by Cauchy's theorem and for sufficiently large $D'>D$, we have
\begin{multline}\label{eq:cauchy theorem log G}
    \left|\int_T^{2T}\log G(D+it,\bm{\pi})dt\right|\le \left|\int_D^{D'}\log G(\sigma+iT,\bm{\pi})d\sigma\right|\\
    +\left|\int_D^{D'}\log G(\sigma+2iT,\bm{\pi})d\sigma\right| +\left|\int_T^{2T}\log G(D'+it,\bm{\pi})dt\right|.
\end{multline} By Lemma \ref{lem:tail of series} and Taylor expansion, we obtain
\begin{equation}\label{eq:log G(sigma+it,pi)}
\log G(\sigma+it,\bm{\pi}) \ll_F (1+1/n_F)^{-\sigma}.
\end{equation}
Substituting \eqref{eq:log G(sigma+it,pi)} into \eqref{eq:cauchy theorem log G}, we have as we take $D'$ to infinity
\begin{align*}
     \left|\int_T^{2T}\log G(D+it,\bm{\pi})dt\right|&\ll_F (1+1/n_F)^{-D}+ T(1+1/n_F)^{-D'} \ll_F 1.
\end{align*} 
Now substitute all estimates back to \eqref{eq:littlewood lemma application}, we have
\begin{align*}
    2\pi\sum_{\substack{T<\gamma_F<2T\\\beta_F>b}}(\beta_F-b) &= \int_T^{2T}\log |G(b+it,\bm{\pi})|dt+O_F(\log T).
\end{align*}
Thus the lemma follows from the definition of $G(s,\bm{\pi})$ in \eqref{def:G}.
\end{proof}

The following lemma establishes a bound for the sum of the horizontal distances of the zeros from the critical line. This result is a direct consequence of Lemma \ref{lem:lemma 6.1} and the strong second-moment bound in \eqref{eq:sharp second moment bound}.

\begin{lem}\label{lem:number of zero-right-half}
 Let $F(s,\bm{\pi})\in \mathcal{B}$ be defined as in \eqref{def:F} satisfying \eqref{assump:assumption 1}{and Hypothesis $\bm{\mathrm{V}}$}. Assume that the set of $L$-functions appearing in the representation \eqref{def:F} of $F(s,\bm{\pi})$ satisfies the strong second-moment bound in \eqref{eq:sharp second moment bound}. Denote the nontrivial zeros of $F(s,\bm{\pi})$ by $\rho_F = \beta_F + i\gamma_F$. Then for any large $T$, we have
    \[
    \sum_{\substack{T<\gamma_F<2T\\\beta_F>1/2}}\Big(\beta_F-\frac{1}{2}\Big) \ll_F T\log\log T.
    \]
\end{lem}
\begin{proof}
    By Lemma \ref{lem:lemma 6.1}, it suffices to prove 
    \[
    \int_T^{2T}\log \Big|F\Big(\frac{1}{2}+it,\bm{\pi}\Big)\Big|dt\ll_F T\log\log T.
    \]
    By definition of $F(s,\bm{\pi})$, we have $F(s,\bm{\pi})$ is bounded above by
    \begin{align*}
         \sum_{j=1}^{M} |c_j|  \prod_{u=1}^V \prod_{l=0}^{k_u}  |L^{(l)}(s, \pi_u)|^{d_{u,l,j}} \le \Big(\sum_{j=1}^{M} |c_j| \Big)\Big(1+\sum_{u=1}^V\sum_{l=0}^{k_u}|L^{(l)}(s,\pi_u)|\Big)^{{\max_{1\le j \le M}}\sum_{u=1}^{V}\sum_{l=0}^{k_u}d_{u,l,j}}.
    \end{align*}
    For any $1 \le u \le V$ and $0 \le l \le k_u$, let $\nu_{u,l} > 0$ and $\nu :=\min_{1 \le u \le V,\,  0 \le l \le k_u}\{\nu_{u,l}\}$. We have
    \begin{align*}
        \quad\int_T^{2T}\log \Big|F\Big(\frac{1}{2}+it,\bm{\pi}\Big)\Big|dt
         &\ll_F \int_T^{2T}\log\Big(1+\sum_{u=1}^V\sum_{l=0}^{k_u}\Big|L^{(l)}\Big(\frac{1}{2}+it,\pi_u\Big)\Big|\Big)dt+T\\
         & \ll_F \frac{1}{\nu}\int_T^{2T}\log\max\Bigg\{1, \max_{\substack{1 \le u \le V \\ 0 \le l \le k_u}}\Big\{\Big|L^{(l)}\Big(\frac{1}{2}+it,\pi_u\Big)\Big|^{\nu_{u,l}}\Big\}\Bigg\}dt+T.
    \end{align*}
    By Jensen's inequality, the main term is
    \begin{align*}
         &\ll_F T\log\Bigg(\frac{1}{T}\int_T^{2T}\max\Bigg\{1, \max_{\substack{1 \le u \le V \\ 0 \le l \le k_u}}\Big\{\Big|L^{(l)}\Big(\frac{1}{2}+it,\pi_u\Big)\Big|^{\nu_{u,l}}\Big\}\Bigg\}dt\Bigg).
    \end{align*}
    It follows that 
    \begin{align*}
        \int_T^{2T}\log \Big|F\Big(\frac{1}{2}+it,\bm{\pi}\Big)\Big|dt \ll_F T \log\Bigg(1+\frac{1}{T}\int_T^{2T}\Big(\sum_{u=1}^{V}\sum_{l=0}^{k_u} \Big|L^{(l)}\Big(\frac{1}{2}+it, \pi_u\Big)\Big|^{\nu_{u,l}}\Big)dt\Bigg)+T.
    \end{align*}
   Therefore, to complete the proof, it suffices to show that for each $1 \le u \le V$ and $0 \le l \le k_u$, there exists $\nu_{u,l} > 0$ such that
    \begin{equation}\label{eq:goal}
    \int_T^{2T}\Big|L^{(l)}\Big(\frac{1}{2}+it,\pi_u\Big)\Big|^{\nu_{u,l}}dt\ll_{\pi_u,l,\eta} T(\log T)^{(2l+\eta)/(4l+1)},
    \end{equation}
    where $\eta$ is the constant assumed in \eqref{eq:sharp second moment bound}.
      
    { We now prove \eqref{eq:goal}. For $l=0$, we take $\nu_{u,0}=2$. Then for any $l \ge 1$, we claim that} 
    \begin{equation}\label{eq:integral of log-derivative}
        \int_{T}^{2T} \Big|\frac{L^{(l)}}{L}\Big(\frac{1}{2}+it, \pi_u\Big)\Big|^{1/(2l)} \,dt  \ll_{\pi_u,l} T \sqrt{\log T}.
    \end{equation}
    Applying H\"older's inequality, we obtain
\[
\int_{T}^{2T} \Big|L^{(l)}\Big(\frac{1}{2}+it, \pi_u \Big)\Big|^{\nu_{u,l}} dt \le \Big(\int_{T}^{2T} \Big|\frac{L^{(l)}}{L}\Big(\frac{1}{2}+it, \pi\Big)\Big|^{p\nu_{u,l}} dt\Big)^{1/p} \Big(\int_{T}^{2T} \Big|L\Big(\frac{1}{2}+it, \pi\Big)\Big|^{q\nu_{u,l}} dt\Big)^{1/q}
\]
with $\nu_{u,l}=\frac{2}{4l+1}$, $p=1+\frac{1}{4l}$ and $q=4l+1$. Applying \eqref{eq:sharp second moment bound} and \eqref{eq:integral of log-derivative}, we have
\begin{equation*}
    \begin{aligned}
        \int_{T}^{2T} \Big|L^{(l)}\Big(\frac{1}{2}+it, \pi_u\Big)\Big|^{\nu_{u,l}} dt &\ll_{\pi_u,l,\eta} (T\sqrt{\log T})^{1/p} (T(\log T)^{\eta})^{1/q}\\
        &\ll_{\pi_u,l,\eta} T(\log T)^{(2l+\eta)/(4l+1)}.
    \end{aligned}
\end{equation*}
This completes the proof of Lemma \ref{lem:number of zero-right-half}.
    
We now prove \eqref{eq:integral of log-derivative}. Let $n \in \mathbb{N}$ be large, $k \in [0, l-1]$ be an integer, and $\rho^{(k)} = \beta^{(k)}+i\gamma^{(k)}$ denote the zeros of $L^{(k)}(s,\pi)$. For $|t-n|\le 1$ and $0<\sigma<1$, we take $F(s,\bm{\pi})=L^{(k)}(s,\pi_u)$ in Lemma \ref{lem:log-derivative} and obtain
\begin{equation}\label{eq:hadamard for L(l+1)/L(l)}
\frac{L^{(k+1)}(s,\pi_u)}{L^{(k)}(s,\pi_u)} = \sum_{|\gamma^{(k)}-n|<2} \frac{1}{s-\rho^{(k)}} + O_{\pi_u,k}(\log n).
\end{equation}
Then, by \eqref{eq:hadamard for L(l+1)/L(l)}, it follows that
\begin{equation}\label{eq:L(l+1)/L(l) bound}
\int_{n-\frac{1}{2}}^{n+\frac{1}{2}}\Big|\frac{L^{(k+1)}}{L^{(k)}}\Big(\frac{1}{2}+it, \pi_u\Big)\Big|^{1/2} dt \ll_{\pi_u,k} \int_{n-2}^{n+2}\Big|\sum_{|\gamma^{(k)}-n|<2} \frac{1}{\frac{1}{2}+it-\rho^{(k)}}\Big|^{1/2} dt + \sqrt{\log n}. 
\end{equation}
Applying \cite[Lemma 4.1]{levinsonmontgomery} and using Lemma \ref{lem:local-zeros} in \eqref{eq:L(l+1)/L(l) bound}, we obtain 
\[
\int_{n-\frac{1}{2}}^{n+\frac{1}{2}}\Big|\frac{L^{(k+1)}}{L^{(k)}}\Big(\frac{1}{2}+it, \pi_u\Big)\Big|^{1/2} dt\ll_{\pi_u,k} \sqrt{\log n}.
\]
From this, we have
\begin{align}
  \int_T^{2T} \Big|\frac{L^{(k+1)}}{L^{(k)}}\Big(\frac{1}{2}+it, \pi_u\Big)\Big|^{1/2} dt \ll_{\pi_u,k} T\sqrt{\log T}.\label{eq:bound for L(l+1)/L(l) interval T to 2T}  
\end{align}

By \eqref{eq:bound for L(l+1)/L(l) interval T to 2T}, the functions
\[
\left|
\frac{L^{(k+1)}}{L^{(k)}}
\left(\frac12+it,\pi_u\right)
\right|^{1/(2l)},
\qquad 0\leq k\leq l-1,
\]
belong to $L^l([T,2T])$. Thus, by \eqref{eq:bound for L(l+1)/L(l) interval T to 2T} and H\"older's inequality, we have
\begin{equation*}
    \begin{aligned}
        \int_T^{2T} \Big|\frac{L^{(l)}}{L}\Big(\frac{1}{2}+it, \pi_u \Big)\Big|^{1/(2l)} dt &= \int_T^{2T} \Big|\frac{L'}{L}\cdot\frac{L''}{L'}\cdots\frac{L^{(l)}}{L^{(l-1)}}\Big(\frac{1}{2}+it, \pi_u\Big)\Big|^{1/(2l)} dt \ll_{\pi_u,l} T\sqrt{\log T},
    \end{aligned}
\end{equation*} 
which completes the proof of \eqref{eq:integral of log-derivative}.
\end{proof}

We can now prove the first half of Theorem \ref{thm:zeros-near-1/2}, that is, the number of zeros to the right of the critical line is small.
\begin{lem}\label{lem:count of zeros on the right half}
 Let $F(s,\bm{\pi})\in \mathcal{B}$ be defined as in \eqref{def:F} satisfying \eqref{assump:assumption 1}{and Hypothesis $\bm{\mathrm{V}}$}. Assume that the set of $L$-functions appearing in the representation \eqref{def:F} of $F(s,\bm{\pi})$ satisfies the strong second-moment bound in \eqref{eq:sharp second moment bound}. Then, for any $\delta > 0$ and any large $T$, we have
\[
N_F^+\Big(\frac{1}{2}+\delta;T, 2T\Big) \ll_{F} \frac{T\log\log T}{\delta}.
\]
\end{lem}
\begin{proof}
From Lemma \ref{lem:number of zero-right-half}, we have
$$
\sum_{ \substack{T<\gamma_F< 2T \\ \beta_F>1/2+\delta}} \Big(\beta_F-\frac{1}{2}\Big) \le \sum_{ \substack{T<\gamma_F<2T  \\ \beta_F>1/2}} \Big(\beta_F-\frac{1}{2}\Big) \ll_F T\log\log T.
$$
On the other hand, we have
$$
\sum_{ \substack{T<\gamma_F<2T \\ \beta_F>1/2+\delta}} \Big(\beta_F-\frac{1}{2}\Big) \ge \delta N_F^+\Big(\frac{1}{2}+\delta;T,2T \Big).
$$
Combining the two inequalities above gives us the lemma.
\end{proof}

Next we turn our attention to the zeros to the left of the critical line. The following lemma provides an asymptotic formula for the sum of the horizontal distances of nontrivial zeros from an arbitrary vertical line.
From now on, we consider a real $b$ such that
\begin{align}\label{eq:upper bound for b}
    b<\min\{E_1,-1-\Real(\mu_{\pi_u}(r)):u=1,2,\cdots,V;r=1,2,\cdots,m_u\},
\end{align}
where $E_1$ is defined in Proposition \ref{prop:zfr1} and $\mu_{\pi_u}(r)$ are the spectral parameters of $L(s,\pi_u)$.

\begin{lem}\label{lem:asymp of zeros of height T}
Let $F(s,\bm{\pi})\in \mathcal{B}$ be defined as in \eqref{def:F} satisfying \eqref{assump:assumption 1}{and Hypothesis $\bm{\mathrm{V}}$}. 
Denote the nontrivial zeros of $F(s,\bm{\pi})$ by $\rho_F = \beta_F + i\gamma_F$. Then, for any fixed $b$ satisfying \eqref{eq:upper bound for b} and any large $T$, 
\begin{align*}
    2\pi\sum_{\substack{T<\gamma_F<2T}}(\beta_F-b) =\deg_{\dim}(F(s,\bm{\pi})) \left(\frac{1}{2}-b\right)T\log T +O_F(T\log\log T). 
\end{align*}
\end{lem}

\begin{proof}
By \eqref{eq:FE-lemma-2} of Lemma \ref{lem:FE}, with the fixed reference index $j_0\in J$ and $e_{u,0}:=e_{u,j_0}$, $\ell_{u,0}:=\ell_{u,j_0}$ from that lemma, for large $|t|$, we have that  $\log|F(b+it,\bm{\pi})|$ equals
\begin{align}\label{eq:log|F(b+it),pi|}
 \log\Big|\prod_{u=1}^V\Big(L(b+it,{\pi}_u)^{e_{u,0}} \prod_{l=0}^{k_u} B(1-b-it,l,\widetilde{\pi}_u)^{d_{u,l,j_0}} \Big)\Big| + O_F(1),
\end{align}
where the $O_F(1)$ absorbs $\log|1+O_F(1/\log t)|$ together with the fixed, $t$-independent constant $\log|\sum_{j\in J}c_j|$ (finite, since $\sum_{j\in J}c_j\ne0$ by \eqref{assump:assumption 1}); no sum over $J$ remains, since \eqref{eq:FE-lemma-2} already collapsed it to a single function. By \eqref{eq:asymp B}, the product appearing in \eqref{eq:log|F(b+it),pi|} equals
\[
\Theta(t) := \prod_{u=1}^V \Big(L(b+it,\pi_u)^{e_{u,0}}\big(m_u\log(1-b-it)\big)^{\ell_{u,0}}\Big)\big(1+O_F(1/\log t)\big),
\]
so that $\log|F(b+it,\bm{\pi})| = \log|\Theta(t)| + O_F(1) =: M(t)+R(t)$, where $M(t):=\log|\Theta(t)|$ and $R(t)=O_F(1)$.

Upon using the definition of $B(s,l,\pi_u)$ as in \eqref{def:B(k,pi
)}, we have
\begin{align}\label{eq:M(t)}
    M(t)
   &= \sum_{u=1}^V e_{u,0}\log|L(b+it,{\pi}_u)|+O_F(\log\log t).
\end{align}
For each $1\le u \le V$, from the functional equation of $|L(b+it, {\pi}_u)|$ in \eqref{eq:FE for L(s,pi)} and the fact that $L(1-b-it, \widetilde{\pi}_u)$ is in the region of absolute convergence, we obtain
\begin{align*}
        \log|L(b+it, {\pi}_u)| &= 
        \sum_{r=1}^{m_u}\Bigg(  \log\Big|\cos\Big(\frac{1-b-it-{\mu_{ {\pi}_u}(r)}}{2}\pi \Big)\Big|\\
        &\quad+\log\Big| \Gamma\Big(\frac{1-b-it+\overline{\mu_{ {\pi}_u}(r)}}{2}\Big)\Big|+\log\Big|\Gamma\Big(\frac{2-b-it-{\mu_{{\pi}_u}(r)}}{2}\Big)\Big|\Bigg)+O_F(1).
    \end{align*}
For each $1 \le r \le m_u$, write $\mu_{\pi_u}(r) := a_{u,r}+ib_{u,r}$.   Thus, for any large $|t|$, we have 
\begin{equation*}
    \begin{aligned}
        \log \left| \Gamma\left( \frac{1 - b - it + \overline{\mu_{\pi_u}(r)}}{2} \right) \right| &= \frac{-b + a_{u,r}}{2} \log \left| \frac{t + b_{u,r}}{2} \right| - \frac{\pi t}{4} + O_F(1);\\
        \log \left| \Gamma\left( \frac{2 - b - it - \mu_{\pi_u}(r)}{2} \right) \right| &= \frac{1-b - a_{u,r}}{2} \log \left| \frac{t + b_{u,r}}{2} \right| - \frac{\pi t}{4} + O_F(1).
    \end{aligned}
\end{equation*}
It follows that
\begin{align}\label{eq:Lof M(t)}
        \log|L(b+it, {\pi}_u)|   = \left(\frac{1}{2}-b\right) m_u\log t+O_F(1).
    \end{align}
    Thus, using \eqref{eq:Lof M(t)} in \eqref{eq:M(t)}, we conclude that 
\[
M(t) = \deg_{\dim}(F(s,\bm{\widetilde{\pi}}))\left(\frac{1}{2}-b\right) \log t
 +O_F(\log\log t).
\]

Combining the bounds for $M(t)$ and $R(t)$ together, we obtain 
\begin{align*}
    \int_T^{2T}\log|F(b+it,\bm{\pi})| \,dt
    &=\deg_{\dim}(F(s,\bm{\widetilde{\pi}})) \Big(\frac{1}{2}-b\Big)T\log T+O_F(T\log\log T).
\end{align*}
Applying Lemma \ref{lem:lemma 6.1}, we conclude that
\begin{align*}
    2\pi\sum_{\substack{T<\gamma_F<2T\\\beta_F>b}}(\beta_F-b) &= \deg_{\dim}(F(s,\bm{\widetilde{\pi}})) \Big(\frac{1}{2}-b\Big)T\log T+O_F(T\log\log T).
\end{align*}
Since $\deg_{\dim}(F(s,\bm{\widetilde{\pi}})) = \deg_{\dim}(F(s,\bm{\pi}))$ and that $b<E_1$, we now obtain Lemma \ref{lem:asymp of zeros of height T}. 
\end{proof}

We now apply the results from Lemmas \ref{lem:number of zero-right-half} and \ref{lem:asymp of zeros of height T} to establish a bound on the number of zeros to the left of the critical line, mirroring the result for the right side.
\begin{lem}\label{lem:count of zeros on the left-half}
Let $F(s,\bm{\pi})$ be defined as in \eqref{def:F} satisfying \eqref{assump:assumption 1}{and Hypothesis $\bm{\mathrm{V}}$}. Assume that the set of $L$-functions appearing in the representation \eqref{def:F} of $F(s,\bm{\pi})$ satisfies the strong second-moment bound in \eqref{eq:sharp second moment bound}. Then, for any $\delta > 0$ and any large $T$, we have
$$
N_F^{-}\Big(\frac{1}{2}-\delta;T,2T\Big) \ll_F \frac{T\log\log T}{\delta}\Big.
$$
\end{lem}

\begin{proof}
Fix a real $b$ satisfying \eqref{eq:upper bound for b}. We can decompose the summation in Lemma \ref{lem:asymp of zeros of height T} as
\begin{multline*}
2\pi\sum_{T<\gamma_F<2T} (\beta_F-b) = 2\pi\sum_{\substack{T<\gamma_F<2T \\ \beta_F>1/2+\delta}} \Big[\Big(\beta_F-\frac{1}{2}\Big) + \Big(\frac{1}{2}-b\Big)\Big] \\
+2\pi\sum_{\substack{T<\gamma_F<2T \\ 1/2-\delta<\beta_F\le1/2+\delta}} \Big[\Big(\beta_F-\frac{1}{2}\Big) + \Big(\frac{1}{2}-b\Big)\Big] +2\pi\sum_{\substack{T<\gamma_F<2T \\ \beta_F\le1/2-\delta}} \Big(\beta_F-b\Big). 
\end{multline*}
By Lemma \ref{lem:number of zero-right-half}, we have
\begin{align}
2\pi\sum_{T<\gamma_F<2T} (\beta_F-b) &\le O_F(T\log\log T) + 2\pi\Big(\frac{1}{2}-b\Big)N_F^+\Big(\frac{1}{2}+\delta;T,2T\Big) \notag\\
&\quad+2\pi\Big(\frac{1}{2}-b\Big)\Big(N_F(T,2T)-N_F^+\Big(\frac{1}{2}+\delta;T,2T\Big)-N_F^-\Big(\frac{1}{2}-\delta;T,2T\Big)\Big) \notag\\
&\quad+2\pi\Big(\frac{1}{2}-\delta-b\Big)N_F^-\Big(\frac{1}{2}-\delta;T,2T\Big).\label{eq:another form in Lemma 6.5}
\end{align}
Applying Theorem \ref{thm:number of zeros}, we have from \eqref{eq:another form in Lemma 6.5} that
\begin{multline}
2\pi\sum_{T<\gamma_F<2T} (\beta_F-b) \le \deg_{\dim}(F(s,\bm{{\pi}})) \Big(\frac{1}{2}-b\Big)T\log T - 2\pi\delta N_F^-\Big(\frac{1}{2}-\delta;T,2T\Big)\\
+ O_F(T\log\log T).\label{eq:asym upper bound in proof of thm 1.5}
\end{multline}
On the other hand, by Lemma \ref{lem:asymp of zeros of height T},
\begin{align}
    2\pi\sum_{T<\gamma_F<2T} (\beta_F-b) 
    &=\deg_{\dim}(F(s,\bm{{\pi}})) \Big(\frac{1}{2}-b\Big)T\log T +O_F(T\log\log T).\label{eq:asym in proof of thm 1.5}
\end{align}
From \eqref{eq:asym upper bound in proof of thm 1.5} and \eqref{eq:asym in proof of thm 1.5}, and the fact that $N_F^-(\frac{1}{2}-\delta;T,2T)$ is nonnegative, we have
\[
N_F^-\Big(\frac{1}{2}-\delta;T,2T\Big) \ll_F \frac{T\log\log T}{\delta}. \qedhere
\]
\end{proof}
\subsection*{Proof of Theorem \ref{thm:zeros-near-1/2}} 
Combining Lemmas \ref{lem:count of zeros on the right half} and \ref{lem:count of zeros on the left-half} gives us Theorem \ref{thm:zeros-near-1/2}.\qed

\subsection*{Proof of Corollary \ref{cor: corollary of thm 1.2}} 
Under Hypothesis  $\bm{\mathrm{A}}$, any $L$-function in the set of
$L$-functions appearing in the representation \eqref{def:F} of $F(s,\bm{\pi})$ satisfies the strong second-moment bound
\begin{equation*}
\int_T^{2T}\Big|L\Big(\frac{1}{2}+it,\pi_u \Big)\Big|^{2}dt \ll_{\pi_u, \varepsilon} T(\log T)^{1+\varepsilon}.
\end{equation*}
The proof can be found, for example, in the work of Milinovich and Turnage-Butterbaugh~\cite[Theorem~1.1]{MT} and Tang and Xiao~\cite[Theorem~1.1]{tangxiao}.  
In addition, for any $L(s,\pi_u) \in \mathcal{R}$, further details can be found in the following references:  
\begin{itemize}
    \item For the Riemann zeta function $\zeta(s)$, see, for example, Hardy and Littlewood \cite{littlewood} and Ingham \cite{ingham}. 
    \item For Dirichlet $L$-functions attached to primitive characters modulo $q$, see, for example, Katsurada and Matsumoto \cite{katsurada}. 
    \item For holomorphic cusp forms of $\Gamma_0(N)$ of weight $k$ with nebentypus $\chi \pmod N$, see, for example, Good \cite{good} and Zhang \cite{zhang}. 
    \item For even Maa{\ss} cusp forms for $\Gamma_0(N)$ with nebentypus $\chi \pmod N$, see, for example, Kuznetsov \cite{kuznetsov} and Zhang \cite{zhang2}. 
\end{itemize}
Combining the strong second-moment bound with Theorem \ref{thm:zeros-near-1/2} completes the proof of the corollary.\qed

\subsection*{Acknowledgements}
The authors are grateful to Jesse Thorner for useful comments and suggestions.  The authors are also grateful to the referees for many valuable suggestions that have greatly increased the clarity and value of the manuscript.

\bibliographystyle{abbrv}
\bibliography{polynomial-zeros}
\end{document}